\documentclass[smallextended,numbook,runningheads]{svjour3}   
\smartqed  % flush right qed marks, e.g. at end of proof
\usepackage{graphicx}

\journalname{BIT}
\date{ \phantom{b} \vspace{45mm}\phantom{e}}
\def\makeheadbox{}%\hspace{-4mm}Version of 27 February 2017}%, file: esep.tex}

\usepackage{amsmath}
\usepackage{amssymb,amsbsy}
\usepackage{amsfonts}
\usepackage{mathrsfs}
\usepackage{bbm}
\usepackage{color}
\newcommand{\vecb}[1]{\boldsymbol#1}
\newcommand{\matr}[1]{\boldsymbol#1}

\newcommand{\dist}{\operatorname{dist}}

\def\Re{\text{\rm Re}\,}

\def\mi{\mathrm{i}}
\def\bone{\mathbbm{1}}
\def\dt{{\tau}}

\def\eps{\epsilon}

\def\C{\mathbb{C}}
\def\R{\mathbb{R}}

\begin{document}

\title{Runge--Kutta convolution coercivity and its use for time-dependent boundary integral equations}
\titlerunning{Runge--Kutta convolution coercivity}

\author{Lehel Banjai \and Christian Lubich}
\institute{
Lehel Banjai \at  The Maxwell Institute for Mathematics in the Sciences; School of Mathematical \& Computer Sciences, Heriot-Watt University, EH14 4AS Edinburgh, UK\\
\email{L.Banjai@hw.ac.uk} 
\and
Christian Lubich \at Mathematisches Institut, Universit\"at T\"ubingen, Auf der Morgenstelle,
  D-72076 T\"ubingen, Germany.\\
  \email{Lubich@na.uni-tuebingen.de}}

\authorrunning{L. Banjai and Ch. Lubich}

\date{}

\maketitle
\begin{abstract} A coercivity property of temporal convolution operators is an essential tool in the analysis of time-dependent boundary integral equations and their space and time discretisations. It is known that this coercivity property is inherited by convolution quadrature time discretisation based on A-stable multistep methods, which are of order at most two. Here we study the question as to which Runge--Kutta-based convolution quadrature methods inherit the convolution coercivity property. It is shown that this holds without any restriction for the third-order Radau IIA method, and on permitting a shift in the Laplace domain variable, this holds for all algebraically stable Runge--Kutta methods and hence for methods of arbitrary order. As an illustration, the discrete convolution coercivity is used to analyse the stability and convergence properties of the time discretisation of a non-linear boundary integral equation that originates from a non-linear scattering problem for the linear wave equation. Numerical experiments illustrate the error behaviour of the Runge--Kutta convolution quadrature time discretisation.

% \PACS{PACS code1 \and PACS code2 \and more}
\end{abstract}

\section{Introduction}
\label{sect:intro}
This paper is concerned with a discrete coercivity property that ensures the stability of time discretisations of boundary integral equations for wave equations, also in situations such as

-   non-linear boundary integral equations;

-  boundary integral equations coupled with a wave equation in an interior domain, with an explicit time discretisation in the  domain.

For convolution quadrature based on A-stable multistep methods (which have approximation order at most two), it is known from \cite{BanLS} that the coercivity property is preserved under time discretisation, uniformly in the temporal stepsize. Here we study the preservation of convolution coercivity under time discretisation by Runge--Kutta convolution quadrature. Up to a shift in the Laplace variable and a corresponding reformulation of the boundary integral equation for an exponentially scaled solution function, we show that the convolution coercivity property is preserved by all convolution quadratures based on algebraically stable Runge--Kutta methods, which include in particular Radau~IIA methods of arbitrary order. Without any such shift and exponential scaling, the convolution coercivity is shown to be preserved by the two-stage Radau~IIA method of order three.

We illustrate the use of the discrete convolution coercivity by the stability and convergence analysis of the Runge--Kutta convolution quadrature time discretisation of a non-linear boundary integral equation for a non-linear scattering problem for the acoustic wave equation. This problem has been studied with different numerical methods in \cite{BanR}.

The discrete convolution coercivity is not needed for the corresponding linear scattering problem, because there the convolution quadrature time discretisation of the linear boundary integral equation can be interpreted as a convolution quadrature discretisation of the convolution operator that maps the data to the solution. Therefore known bounds of the Laplace transform of the solution operator and known error bounds of convolution quadrature yield stability and error bounds \cite{Lub94,Say16}. The same argument can also be used for the coupling of a linear wave equation in an interior domain with the boundary integral equation that describes transparent boundary conditions, provided that the convolution quadrature for the boundary integral equation is based on the same (implicit) time discretisation method as for the wave equation in the interior domain. This precludes explicit time-stepping in the interior. For the coupling of convolution quadrature on the boundary with an explicit time discretisation in the interior, the discrete convolution coercivity as considered in the present paper is needed; see \cite{BanLS,Ebe,KovL} for the coupling of implicit BDF2 convolution quadrature on the boundary with explicit leapfrog time-stepping in the domain for acoustic, elastic and electro-magnetic wave equations, respectively.

The paper is organised as follows:

In Section 2 we recall the continuous-time convolution coercivity, which is related to a coercivity property of the Laplace transform of the (distributional) convolution kernel that holds uniformly for all values of the Laplace-domain frequency variable in a (possibly shifted) right half-plane.

In Section 3 we study the preservation of the convolution coercivity under time discretisation by Runge--Kutta convolution quadrature. This preservation depends on the numerical range of the Runge--Kutta differentiation symbol, which is shown to lie in the right half-plane for algebraically stable Runge--Kutta methods. With a matrix-function inequality that is obtained as an extension of a theorem of von Neumann, we then prove our main result, Theorem~\ref{thm:rk-pos}, which yields the discrete convolution coercivity.

Section 4 recapitulates error bounds of Runge--Kutta convolution quadrature shown in \cite{BanLM}.

In Section 5 we apply our results to the time discretisation of the wave equation with a non-linear impedance boundary condition. We study only semi-discretisation in time, but note that this could be extended to full discretisation with  the techniques of \cite{BanR}. The error behaviour is illustrated by numerical experiments in Section~6. In the numerical experiments it is observed that the convolution quadrature based on the three-stage Radau IIA method performs well even without the shift and exponential scaling, which is more favourable than our theoretical results.

\section{Coercivity of temporal convolutions}\label{section:cont_herglotz}

The following coercivity result is given in \cite{BanLS}, where it is used as a basic result in studying boundary integral operators for the acoustic wave equation; see also \cite{KovL} for Maxwell's equation and \cite{Ebe} for elastic wave equations. The result can be viewed as a time-continuous operator-valued extension of a theorem of Herglotz from 1911, which states that an analytic function has positive real part on the unit disk if and only if convolution with its coefficient sequence is a positive semi-definite operation.

Let $V$ be a complex Hilbert space and $V'$ its dual, and let $\langle\cdot,\cdot\rangle$ denote the anti-duality between $V$ and $V'$.
Let $L(s): V \rightarrow V'$ and $R(s) : V \rightarrow V$ be analytic families of bounded linear operators for $\Re s > \sigma$, continuous for $\Re s\ge\sigma$.  We assume the uniform bounds, with some real exponent $\mu$,
\begin{equation}
  \label{eq:bound_B_s}
  \|L(s)\|_{V' \leftarrow V} \leq M(1+ |s|)^\mu, \quad \|R(s)\|_{V \leftarrow V} \leq M (1+|s|)^\mu, \qquad \Re s > \sigma.
\end{equation}
%For integer $m > \mu+1$, we define the integral kernel 
%\begin{equation}
%  \label{eq:B_kernel_m}
%  B_m(t) = \frac1{2\pi\mi} \int_{\sigma + \mi \mathbb{R}} e^{st} s^{-m} L(s) ds.
%\end{equation}
%For a function $g \in C^m([0,T], V)$ with $g(0) = g'(0)= \dots = g^{(m-1)}(0) = 0$,  we let 
%\begin{equation} \label{Bdt}
%\bigl(L(\partial_t)g \bigr) (t) = \left(\frac{d}{dt}\right)^m \int_0^t B_m(t-\tau) g(\tau) d\tau.
%\end{equation}
%We note that $L(\partial_t)g$ is the distributional convolution of the inverse Laplace transform of $L(s)$ with $g$, which is a continuous function under the stated conditions.
This polynomial bound guarantees that $L$ is the Laplace transform of a distribution $\ell$.
If we write $L(s)=s^k L_k(s)$ with an integer $k>\mu+1$, then the Laplace inversion formula
$$
\ell_k(t) = \frac1{2\pi \mi} \int_{\sigma'+i\R} e^{st}\, L_k(s)\, ds, \quad t\ge 0 \quad (\sigma'>\sigma)
$$
defines a continuous and exponentially bounded function $\ell_k$, which has $\ell$ as its $k$th distributional derivative. We write the convolution with $\ell$ as
$$
u(t) = L(\partial_t)f\,(t) = (\ell*f)(t)=\Bigl(\frac d{dt}\Bigr)^k \int_0^t \ell_k(t-\tau)\, f(\tau)\, d\tau,\quad\ t>0,
$$
for functions $f$ on $[0,T]$ whose extension to $t<0$ by $0$ is $k$ times continuously differentiable.
Similarly we consider the convolution $R(\partial_t) f$.

\begin{theorem}\label{thm:cont_Herglotz} \cite[Lemma 2.2]{BanLS} Let $\alpha\ge 0$.
  In the above situation, the following statements are equivalent:
  \begin{enumerate}
  \item $\Re \langle v, L(s) v\rangle \geq \alpha \|R(s) v\|^2 \qquad \hbox{for all } v \in V, \; \Re s > \sigma$.
  \\
\item $ \displaystyle \int_0^\infty e^{-2\sigma t} \,\Re \langle f(t) , L(\partial_t) f (t) \rangle\, dt \geq \alpha \int_0^\infty e^{-2\sigma t}\, \|R(\partial_t) f(t)\|^2 \,dt$\\ for all  $f \in C^k([0,\infty), V)$ with finite support and $f^{(j)}(0)=0$ for $0\le j<k$, 
%$g(0) =  \dots = g^{(m-1)}(0) = 0$, 
and for all $t \geq 0$.
  \end{enumerate}
\end{theorem}

Property {\it 1.}~is known to be satisfied for  the Laplace transforms of various boundary integral operators for wave equations \cite{BamH,BanLS,KovL,LalS,Say16}, and it is a fundamental property in the study of boundary integral equations for wave equations.

We are interested in time discretisations of the convolution operators $L(\partial_t)$ and $R(\partial_t)$ that preserve this coercivity property. It was shown in \cite{BanLS} that this is achieved by convolution quadrature based on A-stable multistep methods such as the first- and second-order backward differentiation formulae. In Theorem~\ref{thm:rk-pos} below we will show that the coercivity property is also preserved by convolution quadrature based on certain Runge--Kutta methods such as the third-order, two-stage Radau IIA method. For the particular case $\sigma=0$, it will be shown to be preserved for all algebraically stable Runge--Kutta methods.

\section{Preserving coercivity by Runge--Kutta convolution quadrature}

\subsection{Runge--Kutta differentiation symbol and convolution quadrature}
\label{subsec:rkcq}
An $m$-stage Runge--Kutta  discretisation of the initial value problem $y'
= f(t,y)$, $y(0) = y_0$, is given by
\[
\begin{split}
Y_{ni} &= y_n + \dt \sum_{j = 1}^m a_{ij} f(t_n+c_jh,Y_{nj}), \qquad i =
1,\dots,m,\\
y_{n+1} & = y_n + \dt \sum_{j = 1}^m b_j f(t_n+c_jh,Y_{nj}),
\end{split}
\]
where $\dt>0$ is the time step, $t_n = n\dt$, and the internal stages $Y_{ni}$ and grid values $y_n$ are
approximations to $y(t_n+c_i \dt)$ and $y(t_n)$, respectively. 
In the  following we use the notation
\[
\mathscr{A} = (a_{ij})_{i,j = 1}^m, \quad b = (b_1,\dots,b_m)^T,
%\quad c^k = (c_1^k,\dots,c_m^k)^T \quad(k\ge 1),
\quad \bone = (1,1,\dots,1)^T.
\]
We always assume that the Runge--Kutta matrix $\mathscr{A}$ is invertible.

As has been shown in \cite{LubO,SchLL,BanL,BanLM}, and in applications to wave propagation problems further  in \cite{BanK,BanMS,BanR,WanW}, Runge--Kutta methods can be used to construct convolution quadrature methods that enjoy favourable properties.
Here one uses the {\it Runge--Kutta differentiation symbol}
\begin{equation}\label{Delta}
\Delta(\zeta) = \Bigl(\mathscr{A}+\frac\zeta{1-\zeta}\bone b^T\Bigr)^{-1}, \qquad 
\zeta\in\C \hbox{ with } |\zeta|<1.
\end{equation}
This is well-defined for $|\zeta|<1$ if %$\mathscr{A}$ is invertible and if 
$R(\infty)=1-b^T\mathscr{A}^{-1}\bone$ satisfies $|R(\infty)|\le 1$. In fact, the Sherman-Morrison formula then yields
$$
\Delta(\zeta) = \mathscr{A}^{-1} -\frac{\zeta}{1-R(\infty)\zeta} \mathscr{A}^{-1} \bone b^T \mathscr{A}^{-1}.
$$
To formulate the Runge--Kutta convolution quadrature for $L(\partial_t )g$, we formally replace in $L(s)$ the differentiation symbol $s$ by $\Delta(\zeta)/\dt$ and expand the operator-valued matrix function
$$
L\Bigl(\frac{\Delta(\zeta)}\dt \Bigr) = \sum_{n=0}^\infty \matr{W}_n(L) \zeta^n,
$$
where in the case of $L(s):V\to V'$ we have the convolution quadrature matrices $\matr{W}_n(L):V^m \to (V')^m$. 
For the discrete convolution with these matrices we use the notation
$${}{}
\bigl(L(\vecb{\partial}_t^\dt)\vecb{f}\bigr)_n = \sum_{j=0}^n \matr{W}_{n-j}(L) \vecb{f}_j
$$
for any sequence $\vecb{f}=(\vecb{f}_n)$ in $V^m$. For vectors of function values of a function $f:[0,T]\to V$ given as $\vecb{f}_n=\bigl(f(t_n+c_i\dt)\bigr)_{i=1}^m$, the $i$th component of the vector $\bigl(L(\vecb{\partial}_t^\dt)\vecb{f}\bigr)_n$ is considered as an approximation to
$\bigl(L(\partial_t)f\bigr)(t_n+c_i\dt)$.

In particular, if $c_m=1$, as is the case with Radau IIA methods, then
the continuous convolution at $t_{n+1}$ is approximated by the last component of the discrete block convolution:
$$
\bigl(L(\partial_t)f\bigr)(t_{n+1}) \approx e_m^T \bigl(L(\vecb{\partial}_t^\dt)\vecb{f}\bigr)_n,
$$
where $e_m = (0, \dots, 0, 1)^T$ is the $m$th unit vector.
We recall the composition rule 
$$
L_2(\vecb{\partial}_t^\dt)\,L_1(\vecb{\partial}_t^\dt)\vecb{f}=(L_2L_1)(\vecb{\partial}_t^\dt)\vecb{f}.
$$
For $\lambda\in\C$, the convolution quadrature $(\vecb{\partial}_t^\dt-\lambda)^{-1}\vecb{f}$ (which is to be interpreted as $L(\vecb{\partial}_t^\dt)\vecb{f}$ for the multiplication operator $L(s)=(s-\lambda)^{-1}$) contains the internal stages of the Runge--Kutta approximation to the linear differential equation $y'=\lambda y +f$ with initial value $y(0)=0$. 

Results on the order of convergence of this approximation are given in \cite{BanL,BanLM,LubO}. The result of \cite{BanLM}, which is relevant for operators $L(s)$ arising in wave propagation, will be restated and extended to the internal stages in Section~\ref{sec:conv}.

\subsection{Numerical range of the Runge--Kutta differentiation symbol}
We now consider methods that are {\it algebraically stable}:
\begin{itemize}
\item All weights $b_i$ are positive.
\item The symmetric matrix with entries
$b_ia_{ij}+b_ja_{ji}- b_ib_j$ is positive semi-definite.
\end{itemize}
Gauss methods and Radau IIA methods are  widely used classes of methods that satisfy this condition. We refer the reader to \cite{HaiW} for background literature on Runge--Kutta methods and their stability notions.

We consider the weighted inner product on $\C^m$,
\begin{equation}\label{b-inner-product}
(u,v) = \sum_{i=1}^m b_i \overline u_i v_i, \qquad u,v\in \C^m.
\end{equation}
We have the following characterisation.

\begin{lemma} \label{lem:rk-nonneg}
For an algebraically stable Runge--Kutta method and for the $b$-weighted inner product \eqref{b-inner-product},
$$
\Re (w,\Delta(\zeta)w) \ge 0 \quad\hbox{ for all}\ w\in \C^m, |\zeta|<1.
$$
Conversely, if the differentiation symbol of a Runge--Kutta method with positive weights $b_i$ satisfies this inequality, then the method is algebraically stable.
\end{lemma}

\begin{proof} With a different notation, this is shown in \cite[p.\,232, (6.19)]{LubO87}. For the convenience of the reader we include the short proof. 
Since for $v=\Delta(\zeta)w$ we have $(w,\Delta(\zeta)w)=(\Delta(\zeta)^{-1}v,v)$, it suffices to show that 
$$
\Re (v,\Delta(\zeta)^{-1}v) \ge 0 \quad\hbox{ for all}\ v\in \C^m, |\zeta|<1.
$$
We rewrite
$$
\Delta(\zeta)^{-1} = \mathscr{A}+\frac\zeta{1-\zeta}\bone b^T = \mathscr{C}+ \frac12\frac{1+\zeta}{1-\zeta}\bone b^T, \quad\ \hbox{ with}\quad \mathscr{C}= \mathscr{A}-\tfrac12\bone b^T,
$$
and observe that (cf.\cite{HaiL})
\begin{align*}
(\bone b^Tv,v) &= \Bigl| \sum_{i=1}^m b_i v_i \Bigr|^2
\\
2\,\Re(\mathscr{C}v,v) &= \sum_{i,j=1}^m (b_ia_{ij}+b_ja_{ji}-b_ib_j) \bar v_i v_j.
\end{align*}
Since $\Re (1+\zeta)/(1-\zeta)>0$ for $|\zeta|<1$, the result follows.
\qed
\end{proof}

In this paper we will need a stronger positivity property, for which we show the following order barrier and  a positive result for the two-stage Radau IIA method, which is of order 3 and has the coefficients
$$
\mathscr{A} = \begin{pmatrix} 5/12 & -1/12 \\ 3/4 & 1/4
\end{pmatrix}, \quad\  b^T = (3/4,1/4).
$$
\begin{lemma}\label{lem:rk-coerc}
(a) For the two-stage Radau IIA method and for the $b$-weighted inner product \eqref{b-inner-product} and corresponding norm $|\cdot|$ we have
\begin{equation}\label{rk-coerc-1}
\Re (w,\Delta(\zeta)w) \ge \tfrac12 (1-|\zeta|^2) |w|^2 \quad\hbox{ for all }\ w\in \C^m, \ |\zeta|\le 1.
\end{equation}
(b) For none of the Gauss methods with two or more stages and none of the Radau IIA methods with three or more stages, there exists $c>0$ such that for all sufficiently small $\delta>0$,
\begin{equation}\label{rk-coerc}
\Re (w,\Delta(\zeta)w) \ge c\delta \,|w|^2 \quad\hbox{ for all }\ w\in \C^m, \ |\zeta| \le e^{-\delta}.
\end{equation}
\end{lemma}

Clearly, \eqref{rk-coerc-1} implies \eqref{rk-coerc} with $c$ arbitrarily close to 1 for small $\delta$. We further note that the implicit Euler method and the implicit midpoint rule (which are the one-stage Radau IIA and Gauss methods, respectively) also satisfy  \eqref{rk-coerc}.

\begin{proof} (a) For the two-stage Radau IIA method we find
$$
\Delta(\zeta) = \frac12 \begin{pmatrix} 3 & 1 - 4\zeta \\ -9 &\; 5 + 4\zeta \end{pmatrix}.
$$
Denoting the diagonal matrix of the weights by $\mathscr{B}={\rm diag}(3/4,1/4)$, we 
note 
$$
\Re (w,\Delta(\zeta)w) = {\overline w}^T \mathscr{B}^{1/2} \cdot \mathscr{B}^{-1/2}
\tfrac12(\mathscr{B}\Delta(\zeta)+\Delta(\overline\zeta)^T\mathscr{B}) \mathscr{B}^{-1/2} \cdot \mathscr{B}^{1/2}w.
$$
We obtain the hermitian matrix
$$
\mathscr{B}^{-1/2}
\tfrac12(\mathscr{B}\Delta(\zeta)+\Delta(\overline\zeta)^T\mathscr{B}) \mathscr{B}^{-1/2} = \frac12 \begin{pmatrix} 3 & -\sqrt3 (1+2\zeta) \\
-\sqrt3 (1+2\overline\zeta) & 5 + 4\, \Re\zeta \end{pmatrix},
$$
which has the trace $4+2\Re\zeta$ and the determinant
$3(1-|\zeta|^2)$. It follows that both eigenvalues are positive and bounded by $6$, and hence the smaller eigenvalue is bounded from below by
$3(1-|\zeta|^2)/6=(1-|\zeta|^2)/2$. This yields the inequality \eqref{rk-coerc-1}.

(b) The proof uses the W-transformation of Hairer \& Wanner, see \cite[p.\,77]{HaiW}. For each of the $m$-stage Gauss and Radau IIA methods, there exists an invertible real $m\times m$ matrix $W$ with first column $\bone$ such that, with the diagonal matrix $\mathscr{B}$ of the weights $b_i$,
$$
W^T \mathscr{B} W=I_m
$$
or in other words, $W^T \mathscr{B}^{1/2}$ is an orthogonal matrix (with respect to the Euclidean inner product), and
$$
\mathscr{A}= WXW^{-1},
$$
where $X-\tfrac12 e_1e_1^T - \beta_m e_me_m^T$ is a skew-symmetric matrix with $\beta_m=0$ for the Gauss method and $\beta_m>0$ for the Radau IIA method.
We write
\begin{eqnarray*}
&&\Re (w,\Delta(\zeta)w) = \Re \overline w^T \mathscr{B} \Delta(\zeta) w 
\\
&&= \Re \overline w^T  \mathscr{B}^{1/2}\cdot \mathscr{B}^{1/2}W \cdot W^T \mathscr{B}^{1/2} \cdot \mathscr{B}^{1/2} \Delta(\zeta) \mathscr{B}^{-1/2} \cdot  \mathscr{B}^{1/2}W \cdot W^T \mathscr{B}^{1/2} \cdot
\mathscr{B}^{1/2} w 
\\
&&= \Re  \overline w^T  \mathscr{B}^{1/2}\cdot \mathscr{B}^{1/2}W \cdot 
W^T \mathscr{B} \Delta(\zeta) W \cdot W^T \mathscr{B}^{1/2} \cdot
\mathscr{B}^{1/2} w ,
\end{eqnarray*}
where we note
$$
W^T \mathscr{B} \Delta(\zeta) W = W^{-1}\Delta(\zeta) W = (W^{-1}\Delta(\zeta)^{-1} W)^{-1}.
$$
Now, by the definition of $\Delta(\zeta)$ and the above-mentioned property of $W^{-1}\mathscr{A}W=X$ together with $W^Tb=e_1$, the matrix $W^{-1}\Delta(\zeta)^{-1} W$ is the sum of a skew-hermitian matrix plus a rank-1 or rank-2 matrix for Gauss or Radau IIA methods, respectively, and by the Sherman-Morrison-Woodbury formula so is its inverse:
$$
W^T \mathscr{B} \Delta(\zeta) W = Y+Z(\zeta),
$$
where $Y$ is skew-hermitian and $Z(\zeta)$ is of rank 1 or 2 for Gauss or Radau IIA methods, respectively. If $w\ne 0$ is in the null-space of  $Z(\zeta)W^T\mathscr{B}$, which is of codimension 1 or 2 for Gauss or Radau, respectively, then we obtain from the above formulas that
$$
\Re (w,\Delta(\zeta)w) =0
$$
in contradiction to \eqref{rk-coerc}. 
\qed
\end{proof}

As we will show in Theorem~\ref{thm:rk-pos} below, Runge--Kutta convolution quadrature with \eqref{rk-coerc} preserves the coercivity property of Theorem~\ref{thm:cont_Herglotz} for arbitrary abscissa $\sigma \ge 0$, while general algebraically stable methods preserve it in the case $\sigma=0$.
Before we state and prove this theorem in Section~\ref{subsec:proof-rk-pos}, we need an  auxiliary result of independent interest.

\subsection{A matrix-function inequality related to a theorem by von Neumann}
We consider again a complex Hilbert space $V$ and its dual $V'$, with the anti-duality denoted by $\langle \cdot, \cdot \rangle$. On $\mathbb{C}^m$ we consider an inner product $(\cdot,\cdot)$ and associated norm $|\cdot|$. An inner product on $V^m$ and the anti-duality between $V^m$ and $(V')^m$ are induced in the usual way: for Kronecker products $a\otimes u$ and $b\otimes v$ with $a,b\in\mathbb{C}^m$ and $u,v\in V$ one defines
$( a\otimes u, b\otimes v) =(a,b)\, (u,v)$ and extends this to a sequilinear form on $V^m\times V^m$, and in the same way one proceeds for the anti-duality $\langle\cdot,\cdot\rangle$ on $V^m \times (V')^m$.

\begin{lemma} \label{lem:BS-pos}
On the Hilbert space $V$,
let $L(s): V \rightarrow V'$ and $R(s) : V \rightarrow V$ be analytic families of bounded linear operators for $\Re s > \sigma$, continuous for $\Re s \ge \sigma$, such that \eqref{eq:bound_B_s} is satisfied and for some $\alpha\ge 0$,
$$
\Re \langle v, L(s) v\rangle \geq \alpha \|R(s) v\|^2, \qquad \hbox{ for all } v \in V, \; \Re s \geq \sigma.
$$
Let the matrix $S\in \mathbb{C}^{m\times m}$ be such that
$$
\Re (w,Sw) \ge \sigma \, |w|^2, \qquad \hbox{ for all } w\in \mathbb{C}^m.
$$
Then,
$$
\Re \langle \vecb{v}, L(S) \vecb{v} \rangle \geq \alpha \|R(S) \vecb{v}\|^2, \qquad \hbox{ for all } \vecb{v} \in V^m.
$$
\end{lemma}

This result can be viewed as an extension of a theorem of von Neumann \cite{Neu} (see also \cite[p.\,179]{HaiW}), which corresponds to the particular case where $L(s)$ is the identity operator on $V$ (when $V$ is identified with $V'$ with the anti-duality given by the inner product on $V$).

\begin{proof} The proof adapts Michel Crouzeix's proof of von Neumann's theorem as given in \cite[p.\,179\,f.]{HaiW}. Without loss of generality we assume here $\sigma =0$. %and the Euclidean inner product on $\mathbb{C}^m$.

First we note that for a {\it diagonal} matrix $S$ the result holds trivially, and so it does for a {\it normal} matrix $S$, which is diagonalised by a similarity transformation with a unitary matrix.

For a non-normal matrix $S$ we consider the matrix-valued complex function
$$
S(z) = \frac z2 (S+S^*) + \frac 12 (S-S^*)
$$
and we observe that $S=S(1)$ and
$$
\Re (w,S(z)w) = \tfrac12 (\Re z) \,\Re (w,Sw).
$$
Together with the condition on $S$ this shows that the numerical range of $S(z)$ is in the right complex half-plane for $\Re z \ge 0$, and hence all eigenvalues of $S(z)$ have non-negative real part. Therefore, the operator functions $L(S(z))$ and $R(S(z))$ are well-defined for $\Re z \ge 0$.

If $\Re z = 0$, then the matrix $S(z)$ is normal, and hence the desired inequality is valid for $S(z)$ with $\Re z=0$. The function
$$
\varphi(z) = \alpha \| R(S(z)) \vecb{v} \|^2 - \Re \langle \vecb{v}, L(S(z))\vecb{v} \rangle
$$
is subharmonic, since the last term is harmonic as the real part of an analytic function and the first term is the inner product of an analytic function with itself, which is subharmonic (as is readily seen by computing the Laplacian and noting that the real and imaginary parts of the analytic function are harmonic). Hence, by the maximum principle (or its Phragm\'en-Lindel\"of-type extension to polynomially bounded subharmonic functions on the half-plane),
$$
\varphi(1) \le \sup_{\Re z = 0} \varphi(z) \le 0,
$$
which is the desired inequality. 
\qed
\end{proof}

\begin{remark} There is a slightly weaker variant of Lemma~\ref{lem:BS-pos}. We formulate it for $\sigma=0$. Let $L(s): V \rightarrow V'$ and $R(s) : V \rightarrow V$ be analytic families of bounded linear operators for $\Re s > 0$, continuous for $s=0$, such that \eqref{eq:bound_B_s} is satisfied and for some $\alpha\ge 0$,
$$
\Re \langle v, L(s) v\rangle \geq \alpha \|R(s) v\|^2, \qquad \hbox{ for all } v \in V, \; \Re s > 0.
$$
Let the matrix $S\in \mathbb{C}^{m\times m}$ be such that all its eigenvalues either have positive real part or are zero, and
$$
\Re (w,Sw) \ge 0, \qquad \hbox{ for all } w\in \mathbb{C}^m.
$$
Then,
$$
\Re \langle \vecb{v}, L(S) \vecb{v} \rangle \geq \alpha \|R(S) \vecb{v}\|^2, \qquad \hbox{ for all } \vecb{v} \in V^m.
$$
This is proved by continuity, using the previous result for $S+\eps I$ and letting $\eps\to 0$.
\end{remark}

\subsection{Preserving the convolution coercivity under discretisation} \label{subsec:proof-rk-pos}

\begin{theorem}\label{thm:rk-pos} Let the $m$-stage Runge--Kutta method satisfy \eqref{rk-coerc} for some inner product $(\cdot,\cdot)$, as in particular is the case for the two-stage Radau IIA method.  In the situation of Theorem~\ref{thm:cont_Herglotz}, condition 1.\ of that theorem implies, for sufficiently small stepsize $\dt>0$ and with  $\widetilde \sigma = \sigma/c$,
\[
\dt\sum_{n = 0}^\infty e^{-2\widetilde \sigma n\dt} \Re \langle \vecb{f}_n, (L(\vecb{\partial}_t^{\dt})\vecb{f})_n \rangle \geq \alpha \dt\sum_{n = 0}^\infty e^{-2\widetilde \sigma n\dt} \|(R(\vecb{\partial}_t^{\dt})\vecb{f})_n\|^2,
\]
for every sequence $\vecb{f}=(\vecb{f}_n)_{n\ge 0}$ in $V^m$ with finitely many non-zero entries. Moreover, in the case $\sigma=0$  this inequality holds for every algebraically stable Runge--Kutta method, with $\widetilde\sigma=0$ and with respect to the $b$-weighted inner product \eqref{b-inner-product} on $\C^m$.
\end{theorem}

\begin{proof}
The proof uses Parseval's formula and combines Lemma~\ref{lem:BS-pos} with Lemmas \ref{lem:rk-nonneg} and \ref{lem:rk-coerc}.
By \eqref{rk-coerc}, with $\widetilde\sigma=\sigma/c$ and $\rho=e^{-\widetilde\sigma\dt}$  we have with respect to the inner product weighted by the $b_i$ that
\begin{equation}\label{Delta-pos}
\Re \Bigl( w,\frac{\Delta(\rho e^{i\theta})}{\dt}w \Bigr) \ge \sigma |w|^2 \qquad \hbox{ for all } w\in \C^m, \ \theta\in \R.
\end{equation}
We abbreviate
$$
\widehat {\matr{L}}(\theta) = L\Bigl( \frac{\Delta(\rho e^{i\theta})}{\dt} \Bigr)
$$
and similarly $\widehat {\matr{R}}(\theta)$. We denote the Fourier series
$$
\widehat {\vecb{f}}(\theta) = \sum_{n=0}^\infty \rho^n e^{in\theta} \vecb{f}_n.
$$
By Parseval's formula and the definition of the convolution quadrature weights $\matr{W}_n(L)$,
\[
\sum_{n = 0}^\infty  \Re\langle \rho^n \vecb{f}_n , \sum_{j = 0}^n \rho^{n-j} \matr{W}_{n-j} (L)\rho^j \vecb{f}_j\rangle = \frac  1{2\pi}\int_{-\pi}^\pi \Re \langle  \widehat {\vecb{f}}(\theta), \widehat {\matr{L}}(\theta)  \widehat {\vecb{f}}(\theta) \rangle d\theta.
\]
Here \eqref{Delta-pos} used in Lemma~\ref{lem:BS-pos} yields
$$
\Re \langle  \widehat {\vecb{f}}(\theta), \widehat {\matr{L}}(\theta)  \widehat {\vecb{f}}(\theta) \rangle \ge \alpha \|\widehat {\matr{R}}(\theta) \widehat {\vecb{f}}(\theta)\|^2.
$$
Moreover, again by Parseval's formula,
\[
\frac 1{2\pi}\int_{-\pi}^\pi \|\widehat {\matr{R}}(\theta) \widehat {\vecb{f}}(\theta)\|^2 d\theta = \sum_{n = 0}^\infty \rho^{2n} \Bigl\|\sum_{j = 0}^n \matr{W}_{n-j}(R) \vecb{f}_j\Bigr\|^2,
\]
which yields the result. 
\qed
\end{proof}

\section{Error bounds of Runge--Kutta convolution quadrature}
\label{sec:conv}
In this section we restate the result of \cite{BanLM} and extend it to cover the approximation properties of the internal stages, which will be needed in the next section. To avoid restating the list of properties required for the underlying Runge--Kutta method, we state the results just for the Radau IIA methods, which appear to be the practically most important class of Runge--Kutta methods to be used for convolution quadrature.

Let $K(s)$, for  $\Re s > \sigma > 0$, be an analytic family of operators between Hilbert spaces $V$ and $W$ (or just Banach spaces are sufficient here), such
that for some real exponent $\mu$ and $\nu \geq 0$ the operator norm is bounded as follows:
\begin{equation}
\label{K-bounds}
\| K(s)\| \leq M(\sigma) \frac{|s|^\mu}{(\Re s)^\nu } \quad\text{ for all }
\Re s >  \sigma.
\end{equation}

\begin{theorem}\cite[Theorem 3]{BanLM}
\label{thm:err-rkcq} 
Let $K$ satisfy \eqref{K-bounds} and consider the Runge-Kutta 
convolution quadrature based on the Radau IIA method with $m$ stages.
Let $r > \max(2m-1+\mu, 2m-1, m+1)$ and $f \in C^r([0,T],V)$ satisfy 
$f(0) = f'(0) = \ldots = f^{(r-1)}(0) = 0$. Then, there exists $\bar{\dt}> 0$  such that for $0 < \dt \leq \bar{\dt}$ and $t_n=n\dt \in [0,T]$,
\begin{align*}
&\| e_m^T \bigl(K(\vecb{\partial}_t^\dt)\vecb{f}\bigr)_n - \bigl(K(\partial_t)f\bigr)(t_{n+1})  \|
\\
& \qquad \leq C
\,h^{\min(2m-1,m+1-\mu+\nu)}\,\Bigl(\|f^{(r)}(0)\|+\int_0^t \|f^{(r+1)}(\tau)\| d\tau\Bigr).
\end{align*}
The constant $C$ is independent of $\dt$ and $f$, but does depend on $\bar{\dt}$,  $T$, and the constants in  \eqref{K-bounds}.
\end{theorem}

The proof in \cite{BanLM} is readily extended to yield the following error bound for the internal stages. Note that here the full order $2m-1$ is replaced by the stage order plus one, $m+1$. We give the result for $m\ge 2$ stages, so that $m+1\le 2m-1$. (For $m=1$, the implicit Euler method, one can use the previous result.)

\begin{theorem}
\label{thm:err-rkcq-int} 
Let $K$ satisfy \eqref{K-bounds} and consider the Runge-Kutta 
convolution quadrature based on the Radau IIA method with $m\ge 2$ stages.
Let $r > \max(m+1+\mu,  m+1)$ and $f \in C^r([0,T],V)$ satisfy 
$f(0) = f'(0) = \ldots = f^{(r-1)}(0) = 0$. Then, there exists $\bar{\dt}> 0$  such that for $0 < \dt \leq \bar{\dt}$ and $t_n=n\dt \in [0,T]$,
\begin{align*}
&\| \bigl(K(\vecb{\partial}_t^\dt)\vecb{f}\bigr)_n - \bigl(K(\partial_t)f(t_{n}+c_i\dt)\bigr)_{i=1}^m  \|
\\
& \qquad \leq C
\,h^{\min(m+1,m+1-\mu+\nu)}\,\Bigl(\|f^{(r)}(0)\|+\int_0^t \|f^{(r+1)}(\tau)\| d\tau\Bigr).
\end{align*}
The constant $C$ is independent of $\dt$ and $f$, but does depend on $\bar{\dt}$,  $T$, and the constants in  \eqref{K-bounds}.
\end{theorem}

\section{Application to the time discretisation of the wave equation with a non-linear impedance boundary condition}
\label{sec:app}

\subsection{A non-linear scattering problem}
We consider the wave equation on an exterior smooth domain $\Omega^+\subset\R^3$. Following \cite{BanR}, we search for a function $u(\cdot,t) \in H^1(\Omega^+)$ satisfying the  weak form of the wave equation
\begin{equation} \label{wave-eq}
\ddot u = \Delta u \qquad \hbox{ in }\Omega^+
\end{equation}
with zero initial conditions and
with the non-linear boundary condition
\begin{equation}\label{bc}
\partial_\nu^+ u = g(\dot u + \dot u^{inc}) - \partial_\nu^+ u^{inc}  \qquad \hbox{ on }\Gamma,\end{equation}
where $\partial_\nu^+$ is the outer normal derivative on the boundary $\Gamma$ of $\Omega^+$, where $g:\R\to\R$ is a given monotonically increasing function, and where $u^{inc}(x,t)$ is a given solution of the wave equation. The interpretation is that the total wave $u^{tot}=u+u^{inc}$ is composed of the incident wave $u^{inc}$ and the unknown scattered wave $u$.

One approach to solve this exterior problem is to determine the Dirichlet and boundary data from boundary integral equations on $\Gamma$ and then to compute the solution at points of interest $x\in\Omega^+$ from the Kirchhoff representation formula. Here we are interested in the stability and convergence properties of the numerical approximation when the time discretisation in the boundary integral equation and in the representation formula is done by Runge--Kutta convolution quadrature. Since our interest in this paper is the aspect of time discretisation, we will not address the space discretisation by boundary elements, though with some effort this could also be included; cf.~\cite{BanR}.

\subsection{Boundary integral equation and representation formula}

Using the standard notation of the boundary integral operators for the Helmholtz equation $s^2 u-\Delta u=0$ $(\Re s>0)$ as used, for example, in \cite{LalS,Say16} and \cite{BanLS,BanR}, we denote by 
$$
S(s):H^{-1/2}(\Gamma) \to H^1(\R^3\setminus\Gamma) \quad \hbox{ and } \quad
D(s):H^{1/2}(\Gamma) \to H^1(\R^3\setminus\Gamma)
$$ 
the single-layer and double-layer potential operators, respectively, and by $V(s)$, $K(s)$, $K^T(s)$, $W(s)$ the corresponding boundary integral operators that form the Calder\'on operator
\begin{equation} \label{B-def}
B(s) = \begin{pmatrix} sV(s) & K(s) \\ -K^T(s) & s^{-1} W(s)
\end{pmatrix} \ : \ H^{-1/2}(\Gamma) \times H^{1/2}(\Gamma) \to H^{1/2}(\Gamma) \times H^{-1/2}(\Gamma)
\end{equation}
and the related operator
\begin{equation} \label{B-imp}
B_{\mathrm{imp}}(s) = B(s) + \begin{pmatrix} 0 & -\frac12 I \\ \frac12 I & 0
\end{pmatrix},
\end{equation}
where the suffix imp stands for impedance. With these operators, the solution $u$ is determined by first solving, for 
$$
\varphi=-\partial_\nu^+u\quad\text{and}\quad\psi=\gamma^+\dot u
$$ 
(where $\gamma^+$ is the trace operator on $\Omega^+$), the time-dependent boundary integral equation (see \cite{BanR})
\begin{equation}\label{bie}
B_{\mathrm{imp}}(\partial_t) \begin{pmatrix} \varphi \\ \psi \end{pmatrix} +
\begin{pmatrix} 0 \\ g(\psi + {\dot u}^{inc}) \end{pmatrix} =
\begin{pmatrix} 0 \\  \partial_\nu^+ u^{inc} \end{pmatrix}.
\end{equation}
The solution $u$ is then obtained from the representation formula
\begin{equation}\label{rep}
u = S(\partial_t)\varphi + D(\partial_t)\partial_t^{-1}\psi.
\end{equation}
We will address the question as to what are the approximation properties when the temporal convolutions in \eqref{bie} and \eqref{rep} are discretised by Runge--Kutta convolution quadrature. Since this will not turn out fully satisfactory, we will further consider time-differentiated versions of \eqref{bie}.

The following coercivity property was proved in \cite{BanLS}.

\begin{lemma} \cite[Lemma 3.1]{BanLS} \label{lem:calderon-coerc} Let
$\langle\cdot,\cdot\rangle_\Gamma$ denote the anti-duality pairing between $H^{-1/2}(\Gamma) \times H^{1/2}(\Gamma)$ and $H^{1/2}(\Gamma) \times H^{-1/2}(\Gamma)$.
There exists $\beta > 0$ so that the Calder\'on operator \eqref{B-def} satisfies
\[
\Re \left \langle
  \begin{pmatrix}
      \varphi \\ \psi
  \end{pmatrix}, B(s) 
  \begin{pmatrix}
      \varphi \\ \psi
  \end{pmatrix}
\right \rangle_\Gamma
 \geq \beta\,c_\sigma\, \left(\|s^{-1}\varphi\|^2_{H^{-1/2}(\Gamma)} + \|s^{-1}\psi\|^2_{H^{1/2}(\Gamma)}\right)
\]
for  $\Re s \ge \sigma> 0$ and for all $\varphi \in H^{-1/2}(\Gamma)$ and $\psi \in H^{1/2}(\Gamma)$,
with $c_\sigma= \min(1,\sigma^2) \,\sigma$.
 \end{lemma}
Since $B_{\mathrm{imp}}(s)$ differs from $B(s)$ by a skew-hermitian matrix, the same estimate then also holds for $B_{\mathrm{imp}}(s)$. Note that Lemma~\ref{lem:calderon-coerc} yields property {\it 1.} of Theorem~\ref{thm:cont_Herglotz} for the Calder\'on operator $B(s)$ and for the multiplication operator $R(s)=s^{-1}$.

\subsection{Time discretisation by Runge--Kutta convolution quadrature}
Using the notation of Section~\ref{subsec:rkcq}, the boundary integral equation \eqref{bie} is discretised in time with a stepsize $\dt>0$ over a time interval $(0,T)$ with $T=N\dt$ by
\begin{equation}\label{bie-num}
B_{\mathrm{imp}}(\vecb{\partial}_t^\dt) \begin{pmatrix} \vecb{\varphi}^\dt \\ \vecb{\psi}^\dt \end{pmatrix} +
\begin{pmatrix} 0 \\ g(\vecb{\psi}^\dt + {\dot {\vecb{u}}}^{inc}) \end{pmatrix} =
\begin{pmatrix} 0 \\  \partial_\nu^+ \vecb{u}^{inc} \end{pmatrix},
\end{equation}
where $( \vecb{\varphi}^\dt , \vecb{\psi}^\dt)= ( \vecb{\varphi}_n , \vecb{\psi}_n)_{n=0}^{N-1}$ with
$( \vecb{\varphi}_n , \vecb{\psi}_n)=(\varphi_{n,i},\psi_{n,i})_{i=1}^m$
 and $(\varphi_{n,i},\psi_{n,i}) \approx (\varphi(t_n+c_i\dt),\psi(t_n+c_i\dt))$ is the numerical approximation that is to be computed, and $\dot {\vecb{u}}^{inc} = ({\dot {\vecb{u}}}^{inc}_n)_{n=0}^{N-1}$ with ${\dot {\vecb{u}}}^{inc}_n=(\dot u^{inc}(t_n+c_i\dt))_{i=1}^m$. The function $g$ acts componentwise. At the $n$th time step, a non-linear system of equations of the following form needs to be solved:
$$
B_{\mathrm{imp}}\Bigl(\frac{\Delta(0)}\tau\Bigr) \begin{pmatrix} \vecb{\varphi}_n \\ \vecb{\psi}_n\end{pmatrix} +
\begin{pmatrix} 0 \\ g(\vecb{\psi}_n + {\dot {\vecb{u}}}^{inc}_n) \end{pmatrix} = \ldots,
$$
where the dots represent known terms. This has a unique solution, because the eigenvalues of $\Delta(0)=\mathscr{A}^{-1}$ have positive real part, and Lemma~\ref{lem:calderon-coerc} and the monotonicity of $g$ then yield the unique existence of the solution by the Browder--Minty theorem; cf.~\cite{BanR} for the analogous situation for multistep-based convolution quadrature.

As an alternative to \eqref{bie-num}, we further consider the time discretisation of the time-differentiated boundary integral equation:
\begin{equation}\label{bie-num-1}
B_{\mathrm{imp}}(\vecb{\partial}_t^\dt) \begin{pmatrix} \dot{\vecb{\varphi}}^\dt \\ \dot{\vecb{\psi}}^\dt \end{pmatrix} +
\begin{pmatrix} 0 \\ g'(\vecb{\psi}^\dt + {\dot {\vecb{u}}}^{inc})(\dot{\vecb{\psi}}^\dt + {\ddot {\vecb{u}}}^{inc}) \end{pmatrix} =
\begin{pmatrix} 0 \\  \partial_\nu^+ \dot{\vecb{u}}^{inc} \end{pmatrix},
\end{equation}
which is now solved for the approximations $(\dot{\vecb{\varphi}}^\dt, \dot{\vecb{\psi}}^\dt )$ (where the dot is just suggestive notation) to $(\dot\varphi,\dot\psi)$ (where the dot means again time derivative).
Here we define $\vecb{\psi}^\dt = (\vecb{\partial}_t^\dt)^{-1} \dot{\vecb{\psi}}^\dt $ and the same for $\vecb{\varphi}^\dt$. Furthermore, $\ddot {\vecb{u}}^{inc}$ contains the values $\ddot u ^{inc}(t_n+c_i\tau)$.

We can go even further and consider the time discretisation of the twice differentiated boundary integral equation:
\begin{align}\nonumber
&B_{\mathrm{imp}}(\vecb{\partial}_t^\dt) \begin{pmatrix} \ddot{\vecb{\varphi}}^\dt \\ \ddot{\vecb{\psi}}^\dt \end{pmatrix} +
\begin{pmatrix} 0 \\ g'(\vecb{\psi}^\dt + {\dot {\vecb{u}}}^{inc})(\ddot{\vecb{\psi}}^\dt + {\dddot {\vecb{u}}}^{inc}) + g''(\vecb{\psi}^\dt + {\dot {\vecb{u}}}^{inc})\cdot (\dot{\vecb{\psi}}^\dt + {\ddot {\vecb{u}}}^{inc})^2 \end{pmatrix} 
\\
&\hspace{8cm} =
\begin{pmatrix} 0 \\  \partial_\nu^+ \ddot{\vecb{u}}^{inc} \end{pmatrix},
\label{bie-num-2}
\end{align}
where again the dots on the approximation $(\ddot{\vecb{\varphi}}^\dt, \ddot{\vecb{\psi}}^\dt )$ are suggestive notation, and we set $\dot{\vecb{\psi}}^\dt = (\vecb{\partial}_t^\dt)^{-1} \ddot{\vecb{\psi}}^\dt$ and $\vecb{\psi}^\dt = (\vecb{\partial}_t^\dt)^{-2} \ddot{\vecb{\psi}}^\dt $, and the same for $\vecb{\varphi}^\dt$.

Finally, at any point $x\in\Omega^+$ of interest we compute the approximation to the solution value $u(x,t_n+c_i\tau)$ by using the same Runge--Kutta convolution quadrature for discretizing the representation formula \eqref{rep}:
\begin{equation}\label{rep-num}
\vecb{u}^\dt = S(\vecb{\partial}_t^\dt)\vecb{\varphi}^\dt + D(\vecb{\partial}_t^\dt)(\vecb{\partial}_t^\dt)^{-1}\vecb{\psi}^\dt.
\end{equation}
%
%\subsection{Operator bounds}
%In the error analysis we will use the following bounds for the single and double layer operators: for $\Re s\ge\sigma >0$,
%\begin{align}
%&\| S(s) \| _{H^1(\Omega^+)\gets H^{-1/2}(\Gamma)} \le C(\sigma)\, \frac {|s|}{\Re s},
%\\
%&\| D(s) \| _{H^1(\Omega^+)\gets H^{1/2}(\Gamma)} \ \le C(\sigma)\, \frac {|s|^{3/2}}{\Re s}.
%\end{align}
%The first bound is given in the proof of \cite[Prop.\,16]{LalS} and the second bound in the proof of
%\cite[Prop.\,19]{LalS}.
%More interesting for the problem at hand is the pointwise evalution of the single- and double-layer operators. For $x\in\Omega^+$ we define the operators $S_x(s):H^{-1/2}(\Gamma)\to \C$ and 
%$D_x(s):H^{1/2}(\Gamma)\to \C$ by
%$$
%S_x(s)\varphi = (S(s)\varphi)(x) \quad\text{and}\quad D_x(s)\psi = (D(s)\psi)(x).
%$$
%These operators are bounded for $\Re s\ge\sigma >0$ and dist$(x,\Gamma)\ge \delta>0$ by
%\begin{align} \label{Sx}
%&\| S_x(s) \| _{\C\gets H^{-1/2}(\Gamma)} \le C(\sigma,\delta)\, |s|\, e^{-\delta\,\Re s}
%\\
%\label{Dx}
%&\| D_x(s) \| _{\C\gets H^{1/2}(\Gamma)} \ \le C(\sigma,\delta)\,  |s|\, e^{-\delta\,\Re s}.
%\end{align}
%The first bound is proved in \cite[Lemma 6]{BanLM} and the second bound is proved similarly.
%
%

\subsection{Error bounds for the linear case}
We consider first the case of a linear impedance function 
$$
g(\xi)=\alpha\xi \quad\text{ with }\quad \alpha\ge 0. 
$$
Let ${\vecb{u}}^\tau=(\vecb{u}_n)_{n=0}^{N-1}$ with $\vecb{u}_n=(u_{n,i})_{i=1}^m$ be the solution approximation obtained by the discretised representation formula \eqref{rep-num} with either of the discretised boundary integral equations \eqref{bie-num} or  \eqref{bie-num-1} or  \eqref{bie-num-2}. The discretisation is done by Runge--Kutta convolution quadrature based on the Radau IIA method with $m$ stages.

Here we obtain the following optimal-order pointwise error bounds for $x$ bounded away from $\Gamma$.
\begin{proposition} Suppose that in a neighbourhood of the boundary $\Gamma$, the incident wave $u_{inc}$  together with its extension by $0$ to $t<0$ is sufficiently regular. For $x\in\Omega^+$ with $\dist(x,\Gamma)\ge \delta>0$, the following optimal-order  error bound is satisfied in the linear situation described above: for $0\le t_n=n\tau\le T$,
$$
|u_n(x) - u(x,t_n) | \le C(\delta,T)\, \tau^{2m-1}.
$$
\end{proposition}

\begin{proof}
We denote 
$$
B_\alpha(s) = B_{\mathrm{imp}}(s)+
\begin{pmatrix} 0&0 \\ 0&\alpha I\end{pmatrix}.
$$
By Lemma~\ref{lem:calderon-coerc}, $B_\alpha(s)$ is invertible for $\alpha\ge 0$ with the bound, for $\Re s\ge \sigma>0$,
\begin{equation}
  \label{eq:Balpha_bound}
\| B_\alpha(s)^{-1} \| \le C(\sigma) \frac{|s|^2}{\Re s}.  
\end{equation}
The exact solution $u(x,t)$ is given by the representation formula \eqref{rep} with
$$
 \begin{pmatrix}
      \varphi \\ \psi
  \end{pmatrix}
= B_\alpha^{-1}(\partial_t) 
\begin{pmatrix} 0 \\  \partial_\nu^+ u^{inc} -\alpha \dot u^{inc}\end{pmatrix}.
$$
For $x\in\Omega^+$ we define the operators $S_x(s):H^{-1/2}(\Gamma)\to \C$ and 
$D_x(s):H^{1/2}(\Gamma)\to \C$ by
$$
S_x(s)\varphi = (S(s)\varphi)(x) \quad\text{and}\quad D_x(s)\psi = (D(s)\psi)(x).
$$
These operators are bounded for $\Re s\ge\sigma >0$ and dist$(x,\Gamma)\ge \delta>0$ by
\begin{align} \label{Sx}
&\| S_x(s) \| _{\C\gets H^{-1/2}(\Gamma)} \le C(\sigma,\delta)\, |s|\, e^{-\delta\,\Re s}
\\
\label{Dx}
&\| D_x(s) \| _{\C\gets H^{1/2}(\Gamma)} \ \le C(\sigma,\delta)\,  |s|\, e^{-\delta\,\Re s}.
\end{align}
The first bound is proved in \cite[Lemma 6]{BanLM} and the second bound is proved similarly.

We thus have
$$
u(x,t) = (M_x(\partial_t)f)(t) 
$$
with
$$
M_x(s)=(S_x(s),D_x(s)s^{-1})B_\alpha(s)^{-1} 
\quad\text{ and }\quad
f= \begin{pmatrix} 0 \\  \partial_\nu^+ u^{inc} -\alpha \dot u^{inc}\end{pmatrix}.
$$
With the above operator bounds we obtain for $\Re s \ge \sigma >0$ and dist$(x,\Gamma)\ge \delta>0$
\begin{equation}\label{Kx-bound}
\| M_x(s) \|_{\C \gets H^{1/2}(\Gamma) \times H^{-1/2}(\Gamma)} \le C(\sigma,\delta) \, \frac{|s|^3}{\Re s} \,
e^{-\delta\,\Re s}\;.
\end{equation}
By the composition rule, the numerical solution obtained by \eqref{bie-num} and \eqref{rep-num} is given as
$$
{\vecb{u}}^\dt(x) = M_x(\vecb{\partial}_t^\dt)\vecb{f},
$$
where $\vecb{f}$ contains the values of $f$ at the points $t_n+c_i\dt$.
If we take instead \eqref{bie-num-1} or \eqref{bie-num-2}, then we have
$$
{\vecb{u}}^\dt(x) = M_x(\vecb{\partial}_t^\dt)\,(\vecb{\partial}_t^\dt)^{-1}\dot{\vecb{f}}\quad\text{ or }\quad
{\vecb{u}}^\dt(x) = M_x(\vecb{\partial}_t^\dt)\,(\vecb{\partial}_t^\dt)^{-2}\ddot{\vecb{f}},
$$
respectively. In view of \eqref{Kx-bound}, Theorem~\ref{thm:err-rkcq} then yields the result.
\qed
\end{proof}

The situation is different if we consider the $H^1(\Omega^+)$ norm of the error.

\begin{proposition} Suppose that in a neighbourhood of the boundary $\Gamma$, the incident wave $u_{inc}$  together with its extension by $0$ to $t<0$ is sufficiently regular. Then, the following   error bounds are satisfied in the linear situation described above: for $0\le t_n=n\tau\le T$,
$$
\|u_n - u(\cdot,t_n) \|_{H^1(\Omega^+)} \le C(T)\, \tau^{k}
$$
with
\begin{align*}
&k= m+1/2 &\text{if \eqref{bie-num} is used},\\
&k=\min(2m-1,m+3/2) &\text{if \eqref{bie-num-1} is used},\\
&k=\min(2m-1,m+5/2) &\text{if \eqref{bie-num-2} is used}.\\
\end{align*}
\end{proposition}

\begin{proof} 
Consider the Laplace transformed wave equation \eqref{wave-eq}
\begin{equation}
  \label{helm-eq}
\begin{aligned}
  -\Delta \hat u + s^2 \hat u &= 0 & &\text{in }\Omega^+,\\
\partial_\nu^+\hat u -\alpha s \hat u &= \hat f & &\text{on }\Gamma^,
\end{aligned}  
\end{equation}
where  $\hat u$ is the Laplace transform of $u$ and $\hat f$ the Laplace transform of 
\[
f = \partial_\nu^+ u^{inc}-\alpha \dot u^{inc}.
\]
 We will require the estimate, see \cite[Equation (2.9)]{BamH}, 
\begin{equation}
  \label{eq:bd_est}
  \|\partial_\nu^+ \hat u\|_{H^{-1/2}(\Gamma)} \leq C(\sigma) |s|^{1/2} \|\hat u\|_{|s|,\Omega^+},
\end{equation}
with $\Re s \geq \sigma > 0$ and the scaled $H^1$ norm
\[
\|\hat u\|^2_{|s|,\Omega^+} = \|\nabla \hat u\|^2_{L^2(\Omega^+)}+|s|^2 \|\hat u\|^2_{L^2(\Omega^+)}.
\]

Testing \eqref{helm-eq} with $s\hat u$, integrating by parts and taking the real part gives
\[
\Re s \|\hat u\|^2_{|s|,\Omega^+}= -\Re \langle \partial_\nu^+ \hat u, s \gamma^+ \hat u\rangle_\Gamma
\leq  C(\sigma) |s|^{1/2} \|\hat u\|_{|s|,\Omega^+}\|\hat\psi\|_{H^{1/2}(\Gamma)},
\]
where $\hat \psi = s \gamma^+ \hat u$ is the Laplace transform of $\psi$. Making use of $\|\hat u\|_{H^1(\Omega^+)} \leq C(\sigma) \|\hat u\|_{|s|,\Omega^+}$ and the bound \eqref{eq:Balpha_bound} gives
\[
\|\hat u\|_{H^1(\Omega^+)} \leq C(\sigma) \frac{|s|^{5/2}}{(\Re s)^2}\|\hat f\|_{H^{-1/2}(\Gamma)}.
\] The stated result then follows from Theorem~\ref{thm:err-rkcq}.
\qed
\end{proof}

\subsection{Convergence for the non-linear problem}
There are several aspects which make the error analysis of the non-linear problem more intricate:
\begin{itemize}
\item The numerical solution can no longer be interpreted as a mere convolution quadrature for an appropriate operator $K(s)$ acting on the data (i.e., the incident wave).
\item We need to impose regularity assumptions on the solution rather than the data.
\item Convolution coercivity now plays an important role in ensuring the stability of the time discretisation.
\end{itemize}
We assume strict monotonicity of the non-linear function $g:\R\to \R$: there exists $\beta>0$ such that
\begin{equation}\label{g-monotone}
(\xi-\eta)\bigl( g(\xi)-g(\eta) \bigr) \ge \beta\, |\xi-\eta|^2 \quad\text{ for all }\ \xi,\eta\in\R.
\end{equation}
Furthermore, we assume that the pointwise application of $g$ maps $H^{1/2}(\Gamma)$  to $H^{-1/2}(\Gamma)$. As is shown in \cite{BanR} by Sobolev embeddings, this is satisfied if $g(\xi)$ grows at most cubically as $|\xi|\to\infty$.

In the following we write for a stepsize $\tau>0$ and a sequence $\vecb{e}=(\vecb{e_n})_{n=0}^{N-1}$ with $\vecb{e_n}=(e_{n,i})_{i=1}^m$ and $e_{n,i}$ in a Hilbert space $V$
$$
\| \vecb{e} \|_{\ell_2^\dt(0:N;V^m)} = \tau \sum_{n=0}^{N-1} \sum_{i=1}^m \| e_{n,i} \|_V^2.
$$
We denote the numerical solution by ${\vecb{u}}^\tau=(u_{n,i})$ and the corresponding values of the exact solution by $\vecb{u}=(u(t_n+c_i\tau))$, where in both cases $n=0,\dots,N-1$ and $i=1,\dots,m$.

We have the following error bound for the non-linear problem. Here the restriction to the two-stage Radau IIA method stems from Lemma~\ref{lem:rk-coerc}.

\begin{proposition} \label{prop:err-nonlin}
Let the non-linear function $g$ be continuous, strictly monotone and have at most cubic growth. Suppose that the solution $u$ to the problem \eqref{wave-eq}--\eqref{bc} is sufficiently regular. 

Consider the time discretisation \eqref{bie-num} and \eqref{rep-num} by the two-stage Radau IIA convolution quadrature method.
Then, there is $\bar\tau>0$ such that for stepsizes $0<\tau\le\bar\tau$, the error in the boundary values satisfies the bound
\begin{equation}\label{err-boundary}
\| \gamma^+{\vecb{u}}^\tau - \gamma^+\vecb{u} \|_{\ell_2^\dt(0:N;H^{1/2}(\Gamma)^2)} \le C\tau^3,
\end{equation}
and the error in the exterior domain is bounded by
\begin{equation}\label{err-ext}
\| {\vecb{u}}^\tau - \vecb{u} \|_{\ell_2^\dt(0:N;H^{1}(\Omega^+)^2)} \le C\tau^{3/2}.
\end{equation}
The constants $C$ are independent of $\tau$ and $N$ with $0<N\tau\le T$, but depend on $T$.
\end{proposition}

\begin{proof} %{\it (of Proposition \ref{prop:err-nonlin})} 
We eliminate $\varphi$ in the system of boundary integral equations \eqref{bie} to arrive at a boundary integral equation for $\psi$,
\begin{equation} \label{psi-eq}
L(\partial_t)\psi + g(\psi + {\dot u}^{inc})  =
 \partial_\nu^+ u^{inc},
\end{equation}
where
$$
L(s) = s^{-1} \Bigl(W(s) - \bigl(\tfrac12 I - K^T(s)\bigr) V(s)^{-1} \bigl(\tfrac12 I - K(s)\bigr)\Bigr) =
-s^{-1}\mathrm{DtN}^+(s)
$$
with the exterior Dirichlet-to-Neumann operator $\mathrm{DtN}^+(s)$. It follows from Propositions 17 and 18 (and their proofs) in \cite{LalS}  that, for $\Re s \ge \sigma >0$, there exist $C(\sigma)$ and $\alpha(\sigma)>0$ such that
\begin{align}\label{L-bound}
&\| L(s) \|_{H^{-1/2}(\Gamma)\gets H^{1/2}(\Gamma)} \le C(\sigma)\,\frac{|s|}{\Re s},
\\
& 
\label{L-coerc}
\Re\langle \psi, L(s) \psi \rangle \ge \alpha(\sigma) \, \frac{\Re s}{|s|^2} \,\|\psi\|_{H^{1/2}(\Gamma)}^2
\quad\text{ for all }\ \psi\in H^{1/2}(\Gamma),
\end{align}
where $\langle\cdot,\cdot\rangle$ denotes the anti-duality pairing between $H^{-1/2}(\Gamma)$ and $H^{1/2}(\Gamma)$.

Thanks to the composition rule, we can do the same for the numerical discretisation \eqref{bie-num} and reduce the numerical system to an equation for $\vecb{\psi}^\dt$, which is just the convolution quadrature time discretisation of \eqref{psi-eq},
\begin{equation} \label{psi-eq-num}
L(\vecb{\partial_t}^\dt)\vecb{\psi}^\dt + g(\vecb{\psi}^\dt + {\dot {\vecb{u}}}^{inc})  =
 \partial_\nu^+ \vecb{u}^{inc}.
\end{equation}
For the error $\vecb{\eps}=\vecb{\psi}^\dt - \vecb{\psi}$ with $ \vecb{\psi}= \bigl((\psi(t_n+c_i\dt))_{i=1}^m\bigr)_{n=0}^{N-1}$ we then have the error equation
\begin{equation} \label{psi-eq-err}
L(\vecb{\partial_t}^\dt)\vecb{\eps} + g(\vecb{\psi}^\dt + {\dot {\vecb{u}}}^{inc})-g(\vecb{\psi} + {\dot {\vecb{u}}}^{inc})  =
\vecb{d}
\end{equation}
with the defect
$$
\vecb{d} = L(\vecb{\partial_t}^\dt)\vecb{\psi} - \bigl(\bigl(L(\partial_t)\psi(t_{n}+c_i\dt)\bigr)_{i=1}^m\bigr)_{n=0}^{N-1},
$$
which is the convolution quadrature error for $L(\partial_t)\psi$.
By Theorem~\ref{thm:err-rkcq-int} and our assumption of a sufficiently regular $\psi=\gamma^+\dot u$, this is bounded by
$$
\| \vecb{d}_n \|_{H^{-1/2}(\Gamma)} \le C\, \dt^3 \quad\text{ for } \ 0\le n\dt \le T.
$$
Since we can apply the same argument also to spatial derivatives of $\psi$ (in the assumed case of a smooth boundary $\Gamma$), we even have
$$
\| \vecb{d}_n \|_{H^{1/2}(\Gamma)} \le C\, \dt^3.
$$
We test \eqref{psi-eq-err} with $\vecb{\eps}$, multiply  with $e^{-2\widetilde\sigma t}$ with $\widetilde\sigma=1/T$ and integrate from $0$ to $T$. With \eqref{L-coerc} and the Runge-Kutta convolution coercivity as given by Theorem~\ref{thm:rk-pos}, and with the strict monotonicity \eqref{psi-eq-err} we conclude that
\begin{align}\nonumber
\alpha  
&\dt\sum_{n = 0}^{N} e^{-2\widetilde \sigma n\dt} \|((\vecb{\partial}_t^{\dt})^{-1}\vecb{\eps})_n\|_{H^{1/2}(\Gamma)}^2 +\beta \dt\sum_{n = 0}^{N} e^{-2\widetilde \sigma n\dt} \|\vecb{\eps}_n\|_{L_2(\Gamma)}^2 
\\ \label{energy-est}
&\le \dt\sum_{n = 0}^{N} e^{-2\widetilde \sigma n\dt} \langle \vecb{\eps}_n, \vecb{d}_n \rangle
\end{align}
and estimate further
$$
\langle \vecb{\eps}_n, \vecb{d}_n \rangle 
\le \| \vecb{\eps}_n \|_{L_2(\Gamma)} \,\| \vecb{d}_n \|_{L_2(\Gamma)} \le \frac\beta 2  \| \vecb{\eps}_n \|_{L_2(\Gamma)} +\frac1{2\beta} \| \vecb{d}_n \|_{L_2(\Gamma)}^2.
$$
We thus find the stability estimate
$$%\begin{align*}
\| (\vecb{\partial}_t^{\dt})^{-1}\vecb{\eps} \|_{\ell_2^\dt(0:N;H^{1/2}(\Gamma)^2)} + 
\| \vecb{\eps} \|_{\ell_2^\dt(0:N;L_2(\Gamma)^2)} 
 \le C \, \| \vecb{d} \|_{\ell_2^\dt(0:N;L_2(\Gamma)^2)}.
$$%\end{align*}
Since $ (\vecb{\partial}_t^{\dt})^{-1}\vecb{\eps} = \gamma^+{\vecb{u}}^\tau - \gamma^+\vecb{u}$, this proves \eqref{err-boundary}.

Let us denote by $M(s) = S(s)V^{-1}(s):H^{1/2}(\Gamma) \to H^1(\Omega^+)$ the operator that maps Dirichlet data in $H^{1/2}(\Gamma)$ to the corresponding solution 
$\hat u\in H^1(\Omega^+)$ of the Helmholtz equation $s^2\hat u-\Delta \hat u=0$. By \cite[Equation (3.10)]{Say16}, this is bounded for $\Re s\ge \sigma>0$ by
$$
\| M(s) \|_{H^1(\Omega^+) \gets H^{1/2}(\Gamma)} \le C(\sigma) \frac{|s|^{3/2}}{\Re s}. 
$$
We then have
\begin{align*}
&{\vecb{u}}^\tau - \vecb{u} = M(\vecb{\partial}_t^{\dt})\gamma^+{\vecb{u}}^\tau - 
\Bigl(\bigl( M(\partial_t)\gamma^+u(t_n+c_i\dt)\bigr)_{i=1}^2 \Bigr)_{n=0}^{N-1}
\\
&= M(\vecb{\partial}_t^{\dt}) (\gamma^+{\vecb{u}}^\tau - \gamma^+\vecb{u}) + 
 \Bigl( M(\vecb{\partial}_t^{\dt})  \gamma^+\vecb{u} - \Bigl(\bigl( M(\partial_t)\gamma^+u(t_n+c_i\dt)\bigr)_{i=1}^2 \Bigr)_{n=0}^{N-1} \Bigr).
\end{align*}
By Theorem~\ref{thm:err-rkcq-int} and the bound for $M$, the last term is bounded by $O(\tau^{5/2})$ in the $H^1(\Omega^+)$ norm. The first term is only $O(\tau^{3/2})$, since we lose a factor $\tau^{3/2}$ from the $O(\tau^3)$ error bound for $\gamma^+ u$ because of the $O(|s|^{3/2})$ bound of $M(s)$; this follows from Lemma 5.2 in \cite{BanL} and Parseval's identity.
\qed
\end{proof}

In a similar way we obtain the following results for the alternative discretisations \eqref{bie-num-1} and \eqref{bie-num-2}:

%\begin{proposition} \label{prop:err-nonlin-diff}
(i) In addition to Proposition~\ref{prop:err-nonlin}, assume that $g$ has bounded second derivatives. With the discretisation \eqref{bie-num-1} instead of \eqref{bie-num}, the error bound in the $H^{1}(\Omega^+)$ norm improves to $O(\tau^{5/2})$, and the $\ell_2^\dt$ error in a point $x$ bounded away from the boundary $\Gamma$ is at most $O(\tau)$.

(ii) In addition to Proposition~\ref{prop:err-nonlin}, assume that $g$ has bounded second and third derivatives. With the discretisation \eqref{bie-num-2} instead of \eqref{bie-num}, the error bound in the $H^{1}(\Omega^+)$ norm improves to $O(\tau^3)$, and the $\ell_2^\dt$ error in a point $x$ bounded away from the boundary $\Gamma$ is at most $O(\tau^2)$.
%\end{proposition}

The proofs of these error bounds are very similar to that of Proposition~\ref{prop:err-nonlin}, using in addition a discrete Gronwall inequality at the end of the estimation of~$\vecb{\eps}$, and an $O(|s|^3)$ bound for the norm of the operator from $H^{1/2}(\Gamma)\to\C$ that maps Dirichlet data to the solution of the Helmholtz equation $s^2\hat u-\Delta \hat u=0$ at a point $x\in\Omega^+$ bounded away from $\Gamma$, for $s$ in a right half-plane. Since our main concern here is to illustrate the use of the convolution coercivity, we omit the details of these extensions.

\begin{remark} If we set $\widetilde L(s)=L(s+\sigma)$ and $\widetilde \psi(t)=e^{-\sigma t}\psi(t)$ for some $\sigma>0$, then the boundary integral equation \eqref{psi-eq} is equivalent to
\begin{equation} \label{psi-eq-shift}
\bigl(\widetilde L(\partial_t)\widetilde \psi\bigr)(t) + e^{-\sigma t} g(e^{\sigma t}\widetilde\psi (t)+ {\dot u}^{inc}(t))  = e^{-\sigma t} \partial_\nu^+ u^{inc}(t).
\end{equation}
By \eqref{L-coerc}, we then have the coercivity estimate for $\widetilde L(s)$ for all $\Re s \ge 0$ (and not just for $\Re s \ge \sigma$):
$$
\Re\langle \psi, \widetilde L(s) \psi \rangle \ge \alpha(\sigma) \, \frac{\Re s+\sigma}{|s+\sigma|^2} \,\|\psi\|_{H^{1/2}(\Gamma)}^2
\quad\text{ for all }\ \psi\in H^{1/2}(\Gamma).
$$
By Theorem~\ref{thm:rk-pos}, the coercivity estimate for the convolution quadrature approximation of $\widetilde L(\partial_t)\widetilde \psi$ is then obtained for {\it every} algebraically stable Runge-Kutta method (and not just the two-stage Radau IIA method). Hence, by discretising the shifted boundary integral equation \eqref{psi-eq-shift} on an interval $[0,T]$ with shift $\sigma=1/T$, we obtain Runge--Kutta based convolution quadrature time discretisations of arbitrarily high order of convergence (assuming sufficient regularity of the exact solution). We remark that similar shifts are familiar in the convergence analysis of space-time Galerkin methods for time-dependent boundary integral equations \cite{BamH}. As in that case, numerical experiments indicate that implementing the shift may not be necessary in practical computations, although this is not backed by theory.
\end{remark}

\section{Numerical experiments}

\subsection{Scattering by the unit sphere}

In these experiments we let $\Omega^+$ be the exterior of the unit sphere and the trace of the incident wave $u^{inc}$ on the sphere be space independent. As constant functions are eigenfunctions of all the integral operators on the sphere \cite{Ned}, the solution will also be constant in space. The eigenvalue for the combined operator $L(s)$  in \eqref{psi-eq} is given by
\begin{equation}
  \label{eq:Lsphere}
  L(s)\hat\psi = -s^{-1} \mathrm{DtN}^+(s)\hat\psi = \left(1+\frac1s\right)\hat\psi,
\end{equation}
for any $\hat\psi$ constant in space.
This operator will reflect well the behaviour of scattering by a convex obstacle, but not that of a general scatterer. For this reason we concentrate on the corresponding  interior problem with
\begin{equation}
  \label{eq:Lsphere_int}
  L^-(s)\hat\psi = s^{-1} \mathrm{DtN}^-(s)\hat\psi =  \left(-\frac1s+\frac{1+e^{-2s}}{1-e^{-2s}}\right)\hat \psi,
\end{equation}
again for $\hat\psi$ constant.  Treating both these operators as scalar, complex valued functions of $s$, we see that both have a better behaviour than the general operators, see \eqref{L-bound} and \eqref{L-coerc}. Namely
\[
|L(s)| \leq C(\sigma), \qquad \Re L(s) \geq 1
\]
and
\[
|L^-(s)| \leq C(\sigma), \qquad \Re L^-(s) \geq \alpha(\sigma).
\]
As the operator $L(s)$ is too simple, in the numerical experiments we only consider the scalar, non-linear equation
\begin{equation}
  \label{eq:system_scalar}  
L^-(\partial_t)\psi +g(\psi+\dot u^{inc}) = 0.
\end{equation}

Even though these operators are of such a simple form, due to the nonlinearity the exact solution is not available. Nevertheless, a highly accurate solution is not expensive to evaluate and can be used to compute the error in the $\ell^\tau_2$ norm. We have performed the numerical experiments with the following choices of $g$ and $u^{inc}$
\[
g_1(\xi) = \tfrac14 \xi+\xi|\xi|, \qquad g_2(\xi) = \tfrac14 \xi+\xi^3, \qquad u^{inc}(t) = 2e^{-10(t-5/2)^2}
\]
and with final time $T = 6$. Note that $g_1$ is once continuously differentiable whereas $g_2$ is infinitely differentiable. The data $u^{inc}$ is not causal, but it is vanishingly small for $t < 0$ and we have found that this discrepancy has no significant effect on the results. 

\begin{figure}
  \centering
  \includegraphics[width=.5\textwidth]{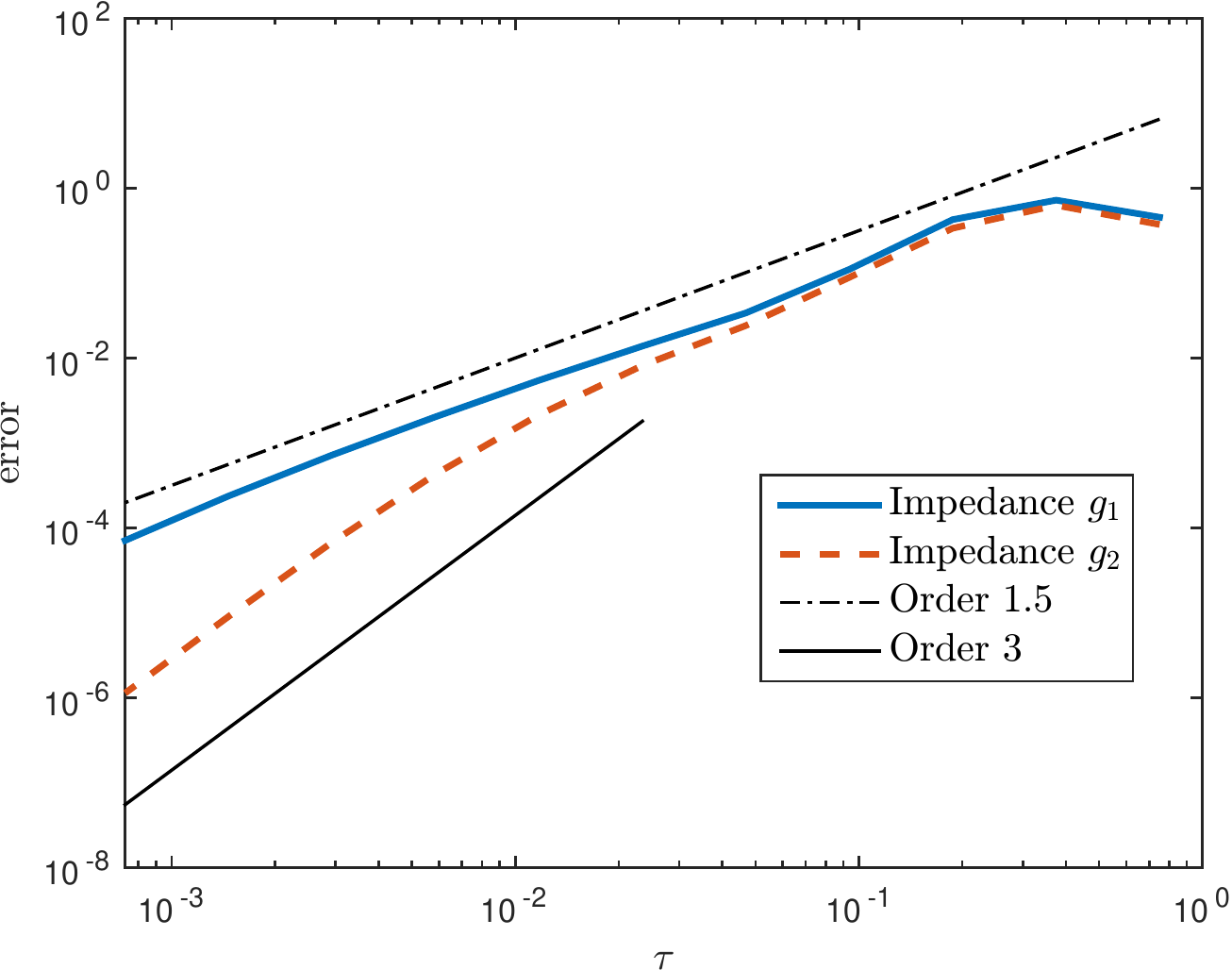}
  \caption{Convergence of the error for the two-stage Radau IIA method with the two different non-linear impedance conditions.}
  \label{fig:sphere}
\end{figure}

In Figure~\ref{fig:sphere} we show the convergence of the two-stage Radau IIA convolution quadrature. As expected, for the smooth non-linear condition we obtain full  order of convergence.  The solution and its first derivative are shown in Figure~\ref{fig:soln_sphere}. Note that the two solutions have a similar shape, but  a closer look at the derivative in Figure~\ref{fig:soln_sphere_blowup} reveals that one is smooth and the other only once continuously differentiable.

\begin{figure}
  \centering
  \includegraphics[width=.45\textwidth]{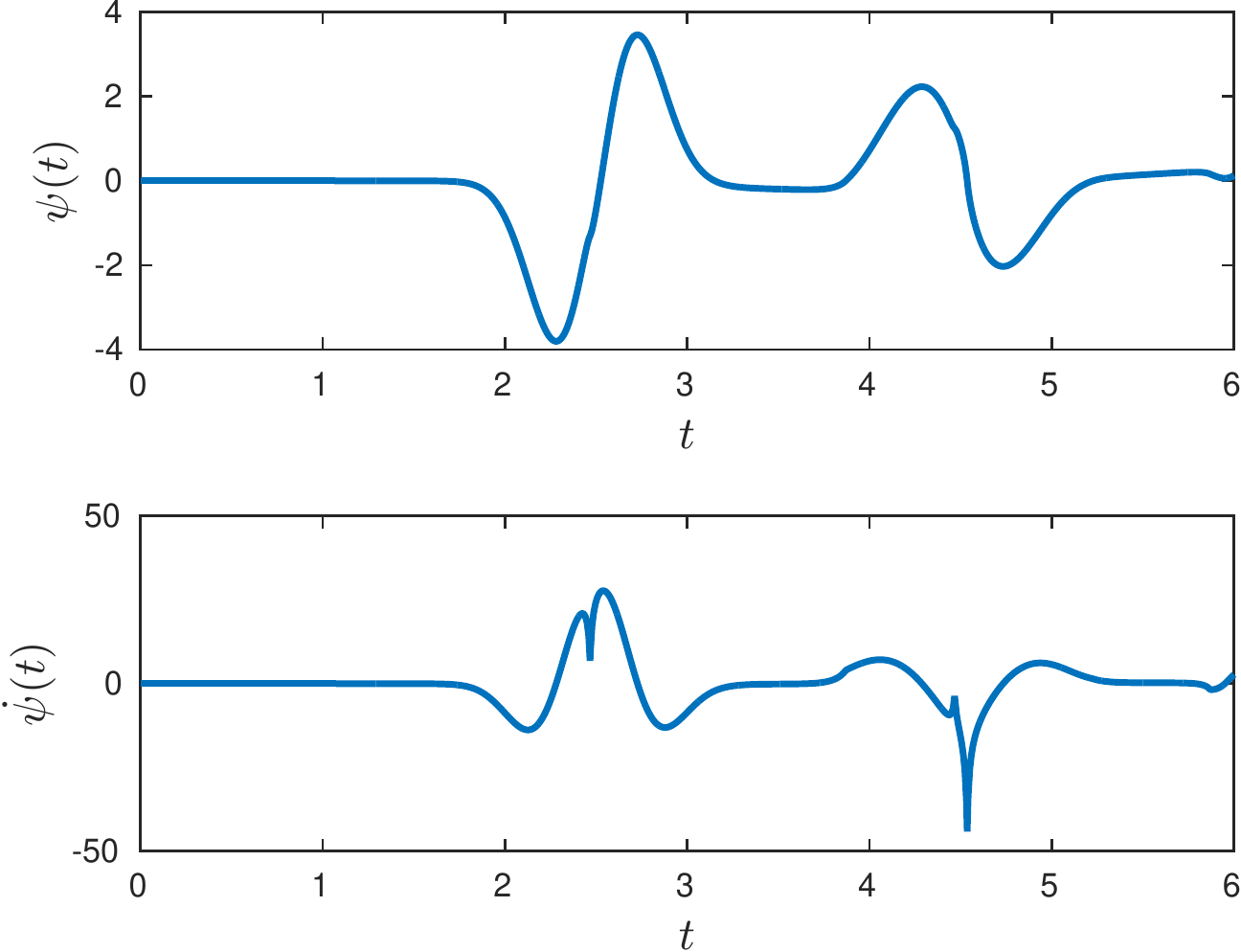}
  \includegraphics[width=.45\textwidth]{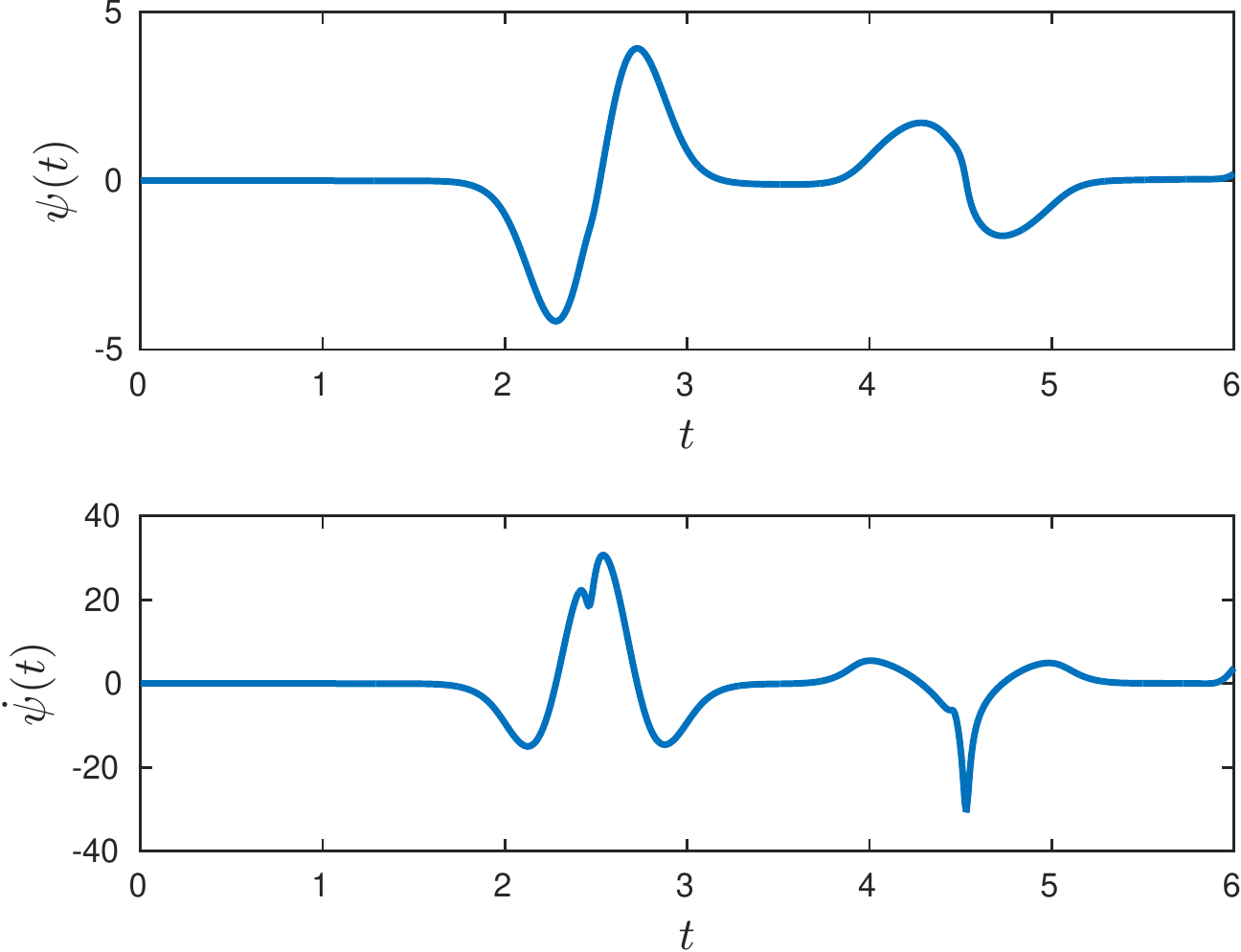}
  \caption{We show the solution and its first derivative. On the left is the solution with the once continuously differentiable impedance  $g_1$ and on the right with the smooth impedance $g_2$.}
  \label{fig:soln_sphere}
\end{figure}

\begin{figure}
  \centering
  \includegraphics[width=.4\textwidth]{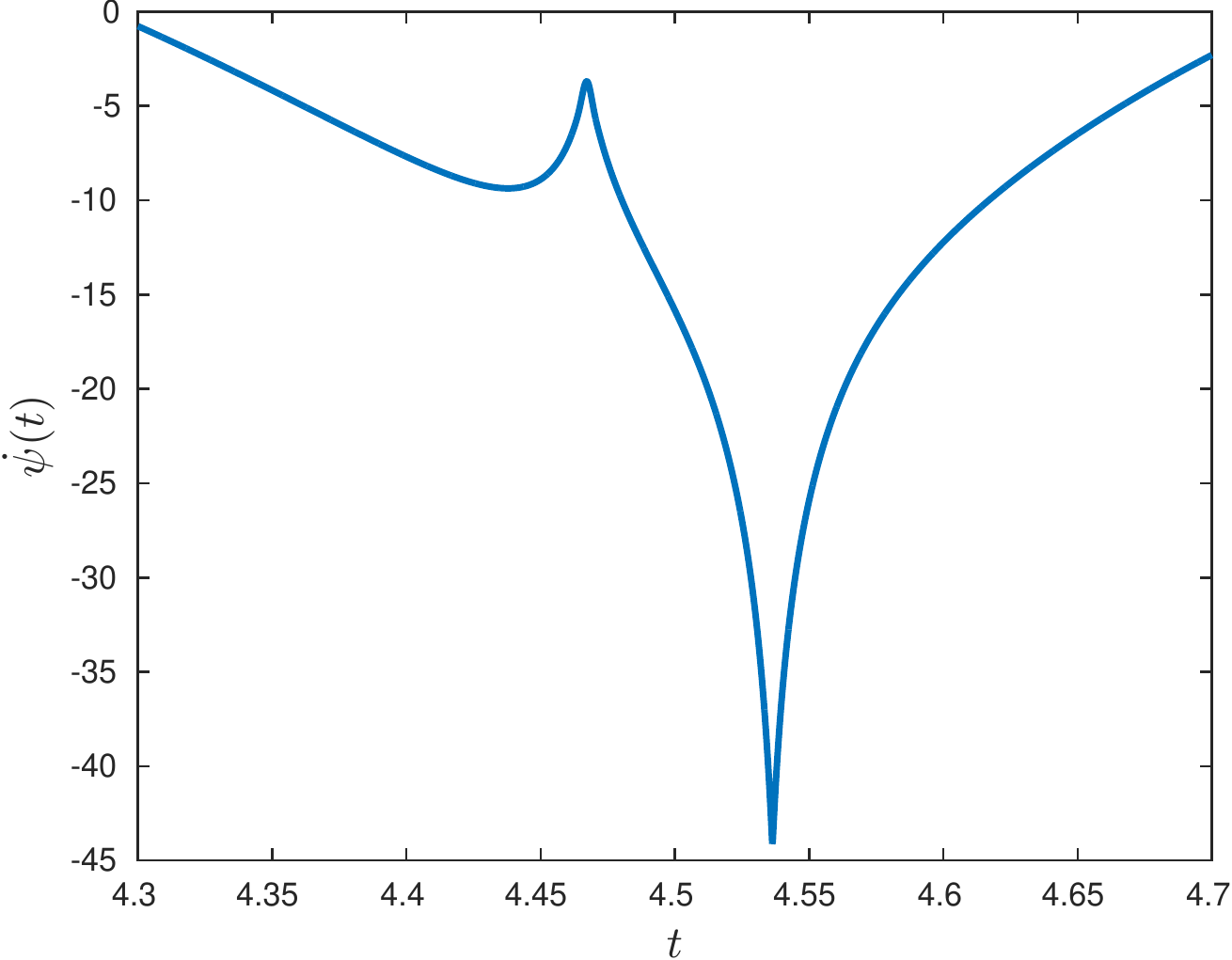}
  \includegraphics[width=.4\textwidth]{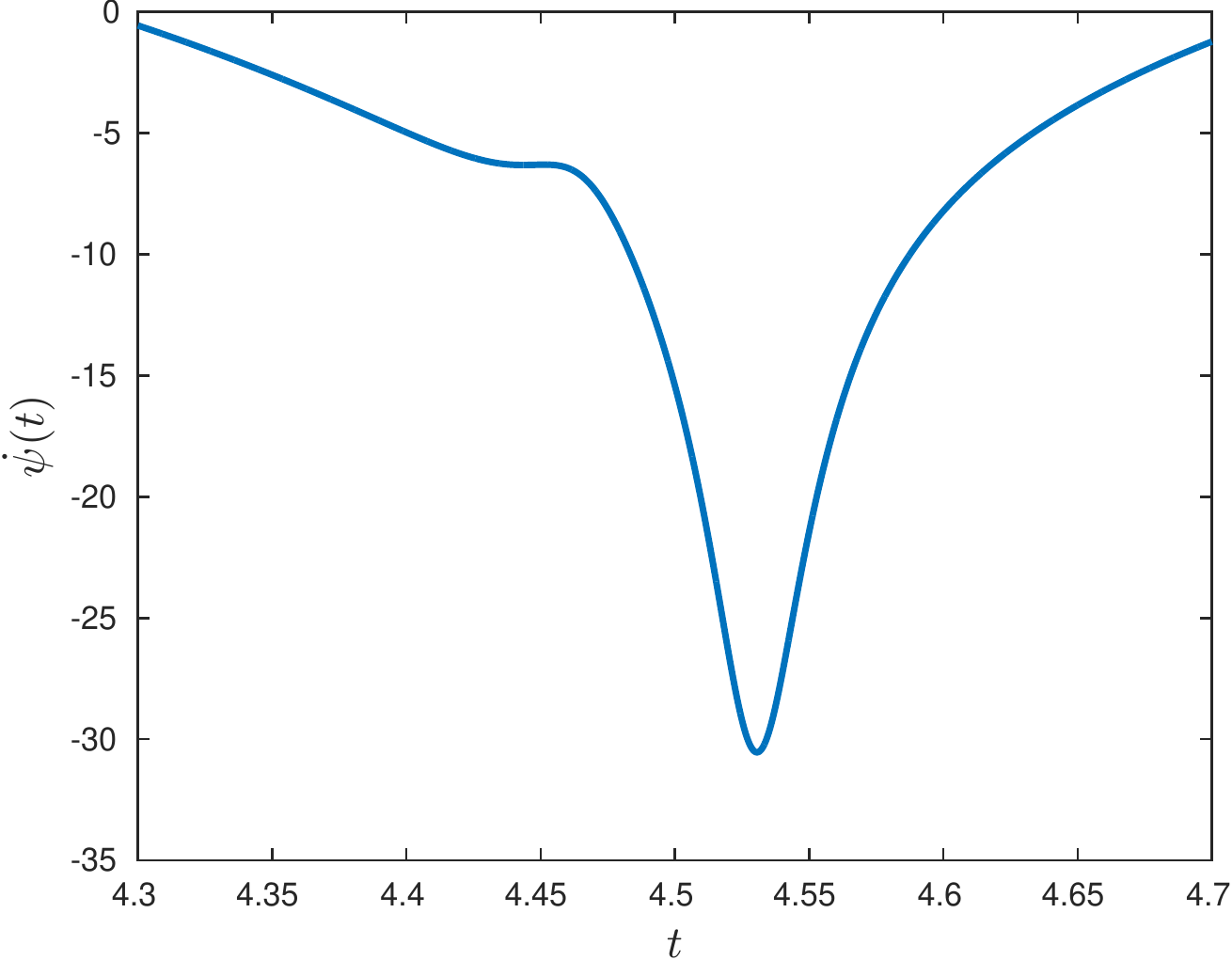}
  \caption{A closer look at the derivatives of the solutions reveals the differing smoothness.}
  \label{fig:soln_sphere_blowup}
\end{figure}

For the interior problem, as $\Re L(s) \geq 1$  the theory also applies to higher order Radau IIA methods. This is however not the case with $L^-(s)$. We nevertheless perform experiments with the three-stage Radau IIA method and obtain good results as shown in Figure~\ref{fig:RK3_sphere}.

\begin{figure}
  \centering
  \includegraphics[width=.5\textwidth]{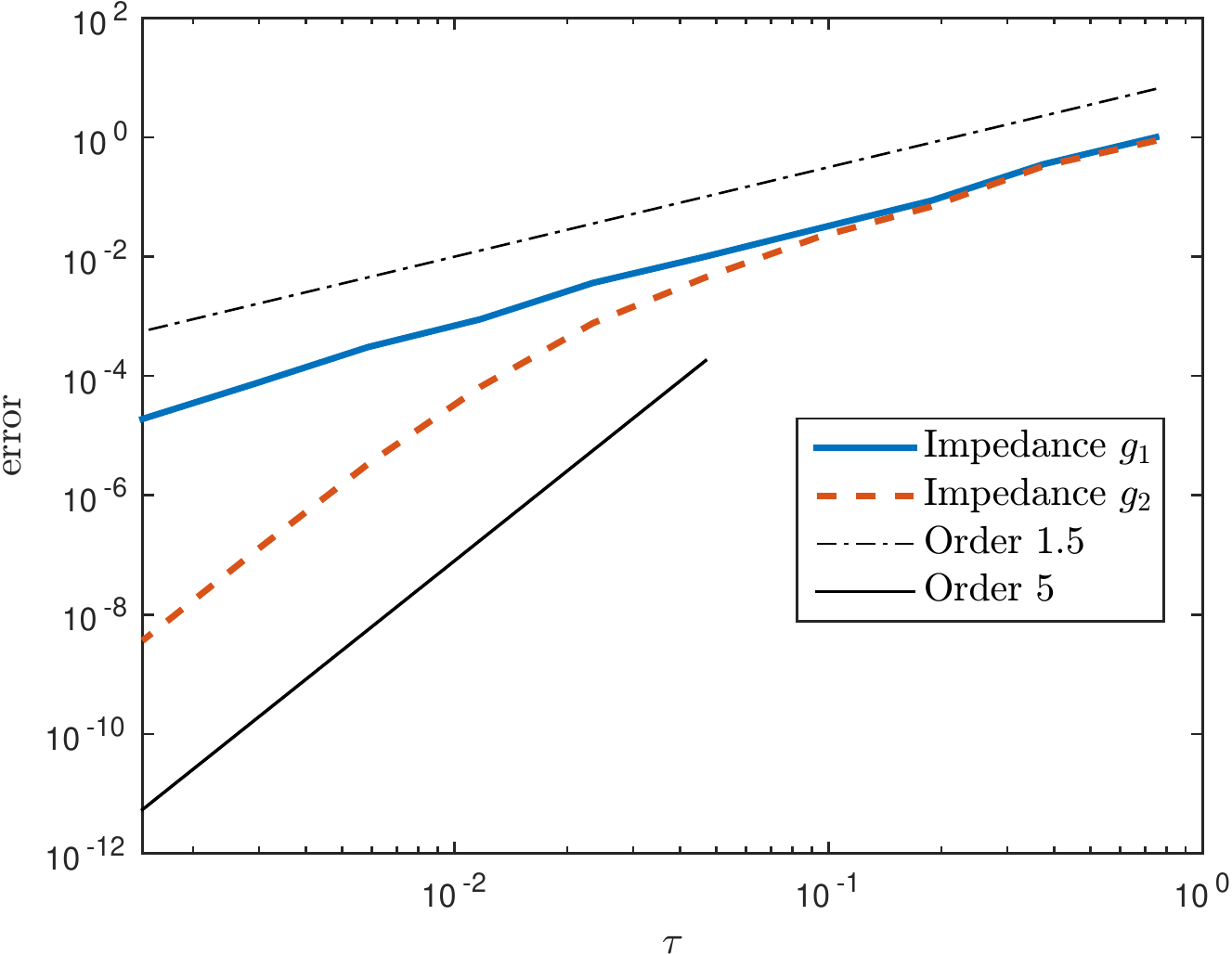}
  \caption{Convergence of the error for the three-stage Radau IIA method with the two different non-linear impedances. }
  \label{fig:RK3_sphere}
\end{figure}

\subsection{A full non-scalar example}

We end the paper with a 2D  example that requires the full BEM discretisation in space. The domain is an L-shape and the incident wave is a plane wave. Piecewise linear boundary element space is used to approximate the Dirichlet trace $\psi$ and piecewise constant boundary element space to approximate the Neumann trace $\varphi$ and the time-discretisation is performed using the two-stage Radau IIA method. The images of the solution are shown in Figure~\ref{fig:Lshape}.

\begin{figure}
  \centering
  \includegraphics[width=.5\textwidth]{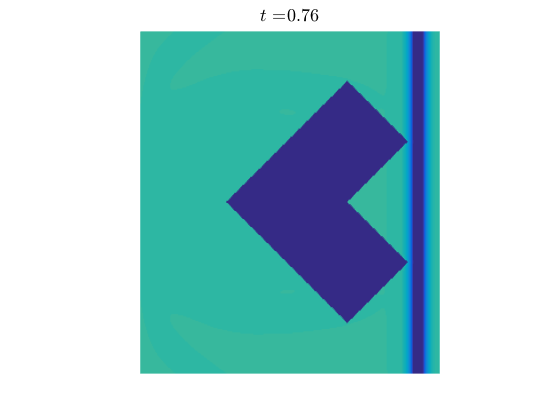}
\hspace{-1.25cm}  \includegraphics[width=.5\textwidth]{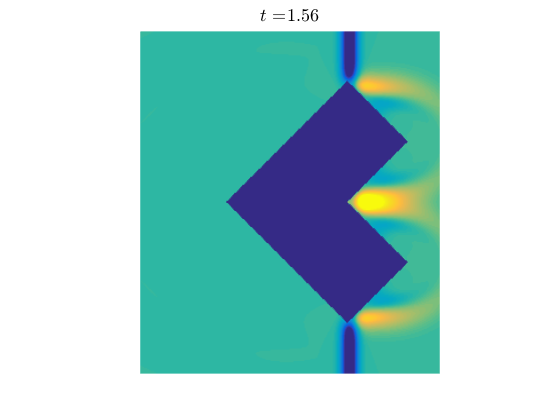}
  \includegraphics[width=.5\textwidth]{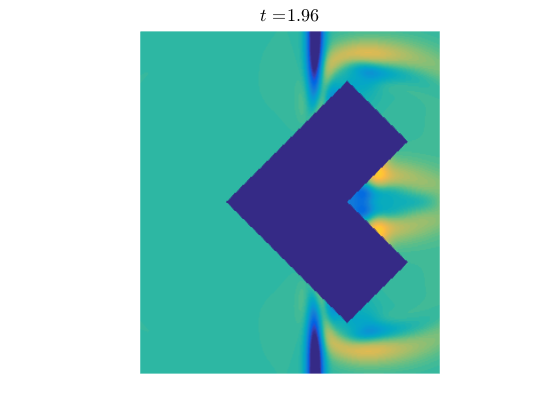}
\hspace{-1.25cm}
  \includegraphics[width=.5\textwidth]{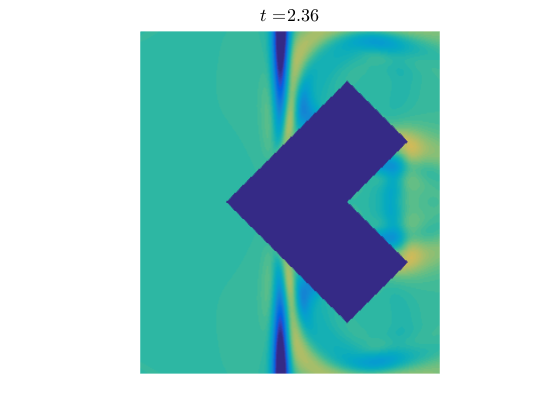}
  \caption{Scattering of a plane wave by an L-shaped domain with a non-linear impedance condition.}
  \label{fig:Lshape}
\end{figure}

%\bibliographystyle{abbrv}
%\bibliography{/home/leki/research/bibs/complete.bib}

\begin{thebibliography}{1}

\bibitem{BamH}
A.~Bamberger and T.~Ha {D}uong.
\newblock Formulation variationelle espace-temps pour le calcul par potentiel
  retard{\'{e}} d'une onde acoustique.
\newblock {\em Math. Meth. Appl. Sci.}, 8:405--435, 1986.

\bibitem{BamH2}
A.~Bamberger and T.~Ha Duong.
\newblock Formulation variationnelle pour le calcul de la diffraction d'une
  onde acoustique par une surface rigide.
\newblock {\em Math. Methods Appl. Sci.}, 8(4):598--608, 1986.

\bibitem{BanK}
L. Banjai and M. Kachanovska. 
Fast convolution quadrature for the wave equation in three dimensions. 
J. Comp. Phys., 279, 103--126, 2014.

\bibitem{BanL}
L.~Banjai and C.~Lubich.
\newblock An error analysis of Runge-Kutta convolution quadrature.
\newblock BIT 51:483--496, 2011.

\bibitem{BanLM}
L.~Banjai, C.~Lubich, and J.~M. Melenk.
\newblock Runge-{K}utta convolution quadrature for operators arising in wave
  propagation.
\newblock {\em Numer. Math.}, 119(1):1--20, 2011.

\bibitem{BanLS}
L.~Banjai, C.~Lubich, and F.-J. Sayas.
\newblock Stable numerical coupling of exterior and interior problems for the
  wave equation.
\newblock {\em  Numer. Math.}, 129(4): 611--646, 2015.

\bibitem{BanMS}
L. Banjai,  M. Messner, and M. Schanz. 
Runge-{K}utta convolution quadrature for the boundary element method. 
Computer Methods in Applied Mechanics and Engineering, 245, 90--101, 2012.

\bibitem{BanR}
L. Banjai and A. Rieder. 
\newblock Convolution quadrature for the wave equation with a non-linear impedance boundary condition.
\newblock arXiv preprint arXiv:1604.05212 (2016).

\bibitem{Ebe}
S. Eberle.
The elastic wave equation and the stable numerical coupling of its interior and exterior problems.
Preprint, Univ.~Tuebingen, {\tt na.uni-tuebingen.de/preprints.shtml}, 2016.

\bibitem{HaiL}
E.~Hairer and C.~Lubich.
\newblock On the stability of Volterra Runge--Kutta methods.
SIAM J. Numer. Anal. 21:123--135, 1984.

\bibitem{HaiW}
E.~Hairer and G.~Wanner. Solving Ordinary Differential Equations II. Stiff and Differential-Algebraic Problems. Springer, 1996.

\bibitem{KovL}
B.~Kov\'acs and C.~Lubich.
Stable and convergent fully discrete interior-exterior coupling of Maxwell's equations.
Preprint,  arXiv:1605.04086 (2016). To appear in Numer. Math.

\bibitem{LalS}
A.~R. Laliena and F.-J. Sayas.
\newblock Theoretical aspects of the application of convolution quadrature to
  scattering of acoustic waves.
\newblock {\em Numer. Math.}, 112(4):637--678, 2009.

\bibitem{Lub94}
C.~Lubich.
\newblock On the multistep time discretization of linear initial-boundary value
  problems and their boundary integral equations.
\newblock {\em Numer. Math.}, 67:365--389, 1994.

\bibitem{LubO87} 
C.~Lubich and A.~Ostermann.
\newblock Multi-grid dynamic iteration for parabolic equations.
\newblock BIT 27:216--234, 1987.

\bibitem{LubO}
C.~Lubich and A.~Ostermann.
\newblock Runge-{K}utta methods for parabolic equations and convolution
  quadrature.
\newblock {\em Math. Comp.}, 60(201):105--131, 1993.

\bibitem{Ned}
J.-C. N{\'e}d{\'e}lec.
\newblock {\em Acoustic and electromagnetic equations}, volume 144 of {\em
  Applied Mathematical Sciences}.
\newblock Springer-Verlag, New York, 2001.
%\newblock Integral representations for harmonic problems.

\bibitem{Neu}
J.~von Neumann.
\newblock Eine Spektraltheorie f\"ur allgemeine Operatoren eines unit\"aren Raumes.
\newblock Math. Nachrichten, 4:258--281, 1951.

\bibitem{Say16} 
F. Sayas. 
Retarded Potentials and Time Domain Boundary Integral Equations: A Road Map. 
Springer, 2016.

\bibitem{SchLL}
A.~Sch{\"a}dle, M.~L{\'o}pez-Fern{\'a}ndez, and C.~Lubich.
\newblock Fast and oblivious convolution quadrature.
\newblock {\em SIAM J. Sci. Comput.}, 28(2):421--438, 2006.

\bibitem{WanW}
X. Wang and D. Weile. 
Implicit Runge-Kutta methods for the discretisation of time domain integral equations. 
IEEE Transactions on Antennas and Propagation, 59(12): 4651--4663, 2011.


\end{thebibliography}

\section*{Acknowledgement}
We thank Ernst Hairer for helpful discussions. This work was partially supported by DFG, SFB 1173.

\end{document}